\documentclass[a4paper,12pt]{article}
\usepackage[utf8]{inputenc}
\usepackage[english]{babel}
\usepackage{enumitem}
\usepackage{color}
\usepackage{amsmath}
\usepackage{amssymb}
\usepackage{amsthm}
\usepackage{mathrsfs}
\usepackage{amsfonts}
\usepackage{epsfig}
\usepackage{graphicx}
\usepackage{epstopdf}
\usepackage{multirow}
\usepackage{cite}
\usepackage{caption}
\usepackage{subfigure}
\newtheorem{theorem}{Theorem}[section]
\newtheorem{definition}{Definition}[section]
\newtheorem{lemma}{Lemma}[section]
\newtheorem{proposition}{Proposition}[section]

\newtheorem{remark}{Remark}[section]
\newtheorem{assumption}{Assumption}[section]
\newtheorem{algorithm}{Algorithm}[section]

\theoremstyle{example}
\theoremstyle{definition}
\theoremstyle{Assumption}

\DeclareMathOperator*{\argmin}{argmin}

\DeclareMathOperator*{\gra}{gra}
\DeclareMathOperator*{\dom}{dom}
\DeclareMathOperator*{\zer}{zer}
\DeclareMathOperator*{\Id}{Id}

\DeclareMathOperator*{\nt}{int}
\DeclareMathOperator{\Pro}{Prob}
\DeclareMathOperator{\prox}{prox}
\allowdisplaybreaks

\begin{document}
	\title{\textbf{Nonlinear forward-backward-half forward splitting with momentum for monotone inclusions}}
	\author{ \sc \normalsize  Liqian Qin$^{a}${\thanks{ email: qlqmath@163.com}},\,\, Yuchao Tang$^{a}${\thanks{email: hhaaoo1331@163.com}}
  ,\,\, Jigen Peng$^{a}${\thanks{ email: jgpeng@gzhu.edu.cn}} \\
		\small $^a$School of Mathematics and Information Science, Guangzhou University,\\
 \small  Guangzhou 510006, China,\\
	}
	\date{}
	\maketitle
	
	\begin{abstract}
{In this work, we propose a new splitting algorithm for solving structured monotone inclusion problems composed of a maximally monotone operator, a maximally monotone and Lipschitz continuous operator and a cocoercive operator. Our method augments the forward–backward-half forward splitting algorithm with a nonlinear momentum term. Under appropriate conditions on the step-size, we prove the weak convergence of the proposed algorithm. A linear convergence rate is also obtained under the strong monotonicity assumption. Furthermore, we investigate a stochastic variance-reduced forward-backward-half forward splitting algorithm with momentum for solving finite-sum monotone inclusion problems. Weak almost sure convergence and linear convergence are also established under standard condition. Preliminary numerical experiments on synthetic datasets and real-world quadratic programming problems in portfolio optimization demonstrate the effectiveness and superiority of the proposed algorithm.
		}
	\end{abstract}
	

\section{Introduction}
In this work, we focus on solving the structured monotone inclusion problems in a real Hilbert space $\mathcal{H}$:
\begin{equation}
\label{ABC}
	\hbox{find} \ x \in \mathcal{H}  \ \hbox{such that} \  0 \in Ax+Bx+Cx,
\end{equation}
where $A: \mathcal{H} \rightarrow 2^{\mathcal{H}}$ is a multi-valued maximally monotone operator, $B: \mathcal{H} \rightarrow \mathcal{H}$ is a single-valued monotone and Lipschitz operator, and  $C: \mathcal{H} \rightarrow \mathcal{H}$ is a cocoercive operator. The monotone inclusion problem not only captures a wide range of convex minimization problems arising in fields such as signal and image processing \cite{Condat2023} and machine learning \cite{{Boyd2011},{Barnert},{Bot2023}}, but also encompasses certain structured monotone inclusion problems as special cases, see for example \cite{{Davis},{FBHF},{Latafat},{Zong}}. In particular, let $\mathcal{H}$ and $\mathcal{G}$ be real Hilbert spaces, and consider the following structured monotone inclusion problems:
\begin{equation}
	\label{ABC2}
	\hbox{find} \ x \in \mathcal{H}  \ \hbox{such that} \  0 \in Ax+L^*(B\Box D \Box E)Lx+Cx,
\end{equation}
where $A: \mathcal{H} \rightarrow 2^{\mathcal{H}}$ is maximally monotone, $C:\mathcal{H} \rightarrow \mathcal{H}$ is $\frac{1}{\beta}$-cocoercive, for some $\beta>0,$ $L:\mathcal{H} \rightarrow \mathcal{G}$ is a bounded linear operator with its adjoint $L^*,$  $B, \ D, \ E:\mathcal{G} \rightarrow 2^{\mathcal{G}}$ are maximally monotone, $D^{-1}$ is $L_D$-Lipschitz, $E^{-1}$ is $\frac{1}{\beta_E}$-cocoercive, and $B\Box D \Box E=(B^{-1}+D^{-1}+E^{-1})^{-1}$ is known as the parallel sum of $B, \ D$ and $E$. An element $x$ is a solution of problem \eqref{ABC2} if and only if there exists $y \in \mathcal{G}$ such that
\begin{equation}
	0 \in \left[
	\begin{array}{cc}
		A & 0 \\
		0 & B^{-1}
	\end{array}
	\right]
	\left[
	\begin{array}{c}
		x\\
		y
	\end{array}
	\right]
	+\left[
	\begin{array}{cc}
		0 & L^* \\
		-L & D^{-1}
	\end{array}
	\right]
	\left[
	\begin{array}{c}
		x\\
		y
	\end{array}
	\right]
	+\left[
	\begin{array}{cc}
		C & 0 \\
		0 & E^{-1}
	\end{array}
	\right]
	\left[
	\begin{array}{c}
		x\\
		y
	\end{array}
	\right].
\end{equation}
This inclusion corresponds to a special case of \eqref{ABC}, since the first operator is maximally monotone, the second is monotone and Lipschitz continuous, and the third is cocoercive. This general framework provides a unified perspective for analyzing and solving diverse problems in optimization and variational analysis.

Given the wide applicability of the structured monotone inclusion problem \eqref{ABC} and its special cases like \eqref{ABC2}, the development of efficient and practical solution algorithms has become a major research focus. Specifically, designing fully splitting algorithms that can process each operator individually and efficiently—multi-valued operators via their resolvents and single-valued operators via explicit evaluations—is critical for leveraging problem structure and enabling scalable implementations.  In the following subsection, we review key related works in this field, focusing on algorithmic approaches for solving monotone inclusion problems, and highlighting the context and motivation behind the contributions presented in this paper.

\subsection{Related works}

\quad In the field of monotone inclusions, researchers are highly motivated to design fully splitting algorithms that allow each operator involved to be processed individually and efficiently, including but not limited to those presented in \cite{{Davis},{Latafat},{Rie},{Ryuu},{Yu},{Zong},{Tang}}. These algorithms aim to exploit the structure of the problem by treating different types of operators with appropriate strategies. Specifically, multi-valued monotone operators are typically handled via their resolvent operators, while single-valued (and often Lipschitz continuous) operators are incorporated through explicit evaluations. This splitting approach not only facilitates the implementation of the algorithms but also enhances their flexibility and scalability when applied to large-scale or structured problems.

\noindent{\bf Forward–backward-half forward splitting algorithm.}
\quad
For monotone inclusion problems involving two operators, many classical algorithms have been proposed and extensively studied \cite{{Malitsky},{Briceno},{Csetnek},{Cevher}}, such as the forward-backward splitting algorithm \cite{{Lions},{Passty}}, the Douglas-Rachford splitting algorithm \cite{DR}, and Tseng's forward-backward-forward splitting algorithm \cite{Tseng}. However, when the problem involves three operators, traditional approaches often resort to combining certain operators to reduce the problem to a two-operator setting. Although theoretically feasible, this strategy typically fails to fully exploit the structural properties of each operator, which may lead to suboptimal efficiency or convergence behavior. In particular, for the monotone inclusion problem considered in this paper,
although $B+C$ inherits monotonicity and Lipschitz continuity, Tseng's method for solving problem \eqref{ABC} suffers from two limitations: (i) requiring two evaluations of each operator per iteration; (ii) failing to leverage the cocoercivity property of $C$. Consequently, designing fully splitting algorithms that can process all three operators separately while leveraging their individual properties remains a significant and active area of research. For problem \eqref{ABC}, Brice\~{n}o et al. \cite{FBHF}
first proposed the forward-backward-half forward (FBHF) splitting algorithm. This algorithm fully exploits the cocoercivity of $C$.
In their approach, while the continuous monotone operator still requires two evaluations per iteration, but reduce the cost for the cocoercive operator to a single evaluation by leveraging its cocoercivity. The FBHF splitting algorithm takes the advantage of the inherent structural properties of individual operators, providing a unified generalization of both Forward-Backward splitting algorithm and Tseng’s method. The detailed FBHF splitting algorithm is defined as follows:
\begin{equation}
		\label{FBHF}
		\left\{
		\begin{array}{lr}
			y_k=J_{\gamma_k A}(x_k-\gamma_k(B+C)x_k), & \\
			x_{k+1}=y_k+\gamma_k(Bx_k-By_k), & \\
		\end{array}
		\right.
\end{equation}
where $ \gamma_k $ is the step-size,  $J_{\gamma_k A}=( \Id+\gamma_k A)^{-1}$ is the resolvent of $A$. They proved the weak convergence of Algorithm \eqref{FBHF} with variable step-size: $\gamma_{k}\in [\eta , \chi-\eta]$, $\eta \in (0,\frac{\chi}{2})$, $\chi=\frac{4 }{\beta+\sqrt{\beta^2+16L^2}}$, and line search procedure. 
Furthermore, they also consider the FBHF splitting algorithm with non-self-adjoint linear operators. The algorithm iterates as follows:
\begin{equation}
	\label{FBHF2}
	\left\{
	\begin{array}{lr}
		y_k=(M_k+A)^{-1}(M_kx_k-(B+C)x_k), & \\
		x_{k+1}=x_k+\lambda_k(M_k(y_k-x_k)+Bx_k-By_k), & \\
	\end{array}
	\right.
\end{equation}
where $M_k: \mathcal{H} \rightarrow  \mathcal{H}$ are bounded linear operators. Under certain boundedness and coercivity conditions on the self-adjoint and skew-symmetric parts of operators $M_k$, as well as constraints on associated Lipschitz and step-size parameters, the weak convergence of the iterative sequence \eqref{FBHF2} is established. Since the introduction of the FBHF splitting algorithm, numerous researchers have explored various extensions, such as inertial \cite{Tang}, relaxed inertial \cite{Zong2024}, multi-step inertial schemes \cite{Zong2023}, Four-operator \cite{Roldan2025a}, and algorithms with deviations \cite{Qin2024}.

\noindent{\bf Warped resolvent.}
The resolvent operator plays a central role in the analysis and design of various operator splitting algorithms. To broaden its scope of applicability, the concept of the nonlinear or warped resolvent, as developed by Giselsson and B$\grave{u}$i-Combettes \cite{bui, Giselsson}, has been introduced. This generalized form of the resolvent has proven effective in modeling a broad range of both existing and novel monotone inclusion algorithms, thereby offering a unified and flexible framework for algorithm design in monotone operator theory. The primary motivation behind introducing warped resolvent lies in leveraging a strategically designed kernel, customized to match the inherent structure of an inclusion problem, to establish streamlined frameworks for devising and analyzing novel splitting algorithms. However, this increased flexibility might require computational trade-offs. The methods proposed in \cite{{bui},{Giselsson},{Latafat}} necessitate an extra corrective projection step to ensure convergence. Later, in \cite{Morin}, the authors proposed a nonlinear forward-backward splitting algorithm with momentum correction. This method  introduces an alternative update correction method that avoids the need for expensive projections, iterates as follows:
\begin{equation}
\label{maincite}
\aligned
&x_{k+1}=(M_k+A)^{-1}(M_kx_k-Cx_k+\gamma_k^{-1}u_k), \\
&u_{k+1}=(\gamma_kM_k-S)x_{k+1}-(\gamma_kM_k-S)x_k.
\endaligned
\end{equation}
where $M_k: \mathcal{H} \rightarrow \mathcal{H}$ are general nonlinear operators, $\gamma_k >0,$ and $S$ is a bounded linear, self-adjoint, strongly positive operator. For the convergence analysis, it is assumed that $\gamma_kM_k-S$ is $L_k$-Lipschitz continuous w.r.t. $S$ for some $L_k \geq 0,$ $\forall k \in \mathbb{N}.$ Under this assumption, $M_k$ is Lipschitz continuous w.r.t. $S$, maximally monotone and strongly monotone w.r.t. $S$ for all $ k \in \mathbb{N}.$ Therefore, the warped resolvent of $A$ with kernel $M_k,$ i.e., $(M_k+A)^{-1} \circ M_k,$ is viable for all $ k \in \mathbb{N}$  \cite{bui}. In \cite{Bredies}, the authors introduced the concept of an admissible preconditioner. For an operator $A: \mathcal{H} \rightarrow 2^{\mathcal{H}},$ a bounded linear operator $M$ that is self-adjoint and positive semidefinite is called an admissible preconditioner if $(M+A)^{-1}M$ is single valued with full domain. The remarkable advantage of nonlinear resolvents is their potent modeling capabilities. This enables a comprehensive and unified perspective on an extensive range of algorithms. In fact, Algorithm \eqref{maincite} is demonstrated by deriving a wide range of special cases, such as the forward-backward splitting algorithm, the forward-reflected-backward splitting algorithm \cite{Malitsky}, etc. Recently,  Rold\'{a}n and Vega \cite{Roldan2025} proposed an inertial and relaxation extension of \eqref{maincite}.
To the best of our knowledge, there is no existing work that considers the FBHF splitting algorithm with momentum, as introduced by \cite{Morin}.

\vskip 5mm

\noindent{\bf Variance reduction stochastic method.}
\quad
Within optimization frameworks, monotone inclusion problems are often formulated as finite-sum problems. For example, finite-sum minimization is ubiquitous in machine learning, where one typically minimizes the empirical risk \cite{LJC}. Stochastic optimization plays a central role in this context. In recent years, a major advancement in this field has been the development of variance-reduction (VR) techniques for stochastic algorithms.
It is important to note that the existing results for solving problem \eqref{ABC} are limited to deterministic methods and do not incorporate stochastic components.  However, evaluating the operator $\sum_{i=1}^N B_ix$ can  become computationally prohibitive in high-dimensional settings. To alleviate this  significant computational burden, we consider the problem \eqref{ABC} under the assumption that the operator $B$ has a finite-sum structure. VR techniques are particularly well-suited to exploit this structure and improve algorithmic efficiency. As a result, numerous stochastic algorithms incorporating variance reduction have been developed for solving  finite-sum monotone inclusion problems \cite{{Cai},{Caixufeng}, {Qin2025},{Sadiev},{Su},{Zhang}}.
To effectively address the computational challenges posed by high-dimensional data in monotone inclusion problems with a finite-sum structure, variance-reduction (VR) techniques have emerged as a powerful tool. Stochastic optimization serves as the backbone of modern machine learning, with stochastic gradient descent (SGD) being one of the pioneering algorithms in this domain \cite{SGD}. Building on this foundation, VR methods such as SAGA \cite{SAGA}, the stochastic average gradient method (SAG) \cite{SAG1, SAG2}, and the stochastic dual coordinate ascent method (SDCA) \cite{ShalevShwartz2013} have fundamentally reshaped the landscape of stochastic optimization by achieving provably faster convergence rate than SGD \cite{Gower}. Among these, the stochastic variance-reduced gradient method (SVRG) \cite{SVRG} stands out as a representative approach that leverages control variates to reduce gradient variance. More recently, Kovalev et al. \cite{KHR} introduced a loopless variant of SVRG, which eliminates the outer loop in the original algorithm and instead applies a probabilistic mechanism to update the full gradient, further improving practicality and efficiency.

Building upon these developments, Alacaoglu et al.  \cite{Alacaoglu-Malitsky} recently proposed a loopless variant of the extragradient method that incorporates variance-reduction techniques for solving variational inequalities in a finite dimensional Euclidean space. Specifically, they consider the following problem:
\begin{equation}
	\label{VI}
	\hbox{find} \ x \in \mathcal{Z}  \ \hbox{such that} \ \langle F(x) , z-x\rangle +g(z)-g(x) \geq 0, \ \forall z \in \mathcal{Z},
\end{equation}
where  $ \mathcal{Z}$ is a finite dimensional vector space, $F$ is a monotone operator, and $g$ is a proper convex, and lower semicontinuous function.  The algorithm iterates as follows:
\begin{equation}
	\label{maincitealg}
	\left\{
	\begin{array}{lr}
	\bar{x}_k= \alpha x_k+(1-\alpha)w_k, & \\
	z_{k+\frac{1}{2}}= \prox_{\tau g}(\bar{x}_k-\tau F(w_k)),& \\
	\hbox{Draw an index} \ \xi_k \ \hbox{according to} \ Q,& \\
	z_{k+1}=\prox_{\tau g}(\bar{x}_k-\tau(F(w_k)+F_{\xi_k}(z_{k+\frac{1}{2}})-F_{\xi_k}(w_k))),&\\
	w_{k+1}=
	\begin{cases}
		x_{k+1},&\hbox{with probability} \,\,p,\\
		w_k,&\hbox{with probability} \,\,1-p,
	\end{cases}
	\end{array}
	\right.
\end{equation}
where $\alpha \in (0,1),$ $\tau$ is the step-size parameter, $Q$ is a probability distribution, and  $p \in (0,1).$ They have proved the almost sure convergence of the Algorithm \eqref{maincitealg} in a finite dimensional Euclidean space. The authors further extended this approach to Tseng's method \cite{Tseng} and forward-reflected-backward splitting algorithm \cite{Malitsky} for solving monotone inclusion problems involving two operators in finite dimensional Euclidean space:
\begin{eqnarray*}
\mbox{find} \ x \in \mathbb{R}^d \ \ \mbox{such that} \ 0\in Ax+Bx,
\end{eqnarray*}
where $A: \mathbb{R}^d \rightarrow 2^{\mathbb{R}^d}$ is a set-valued maximally monotone operator and $B: \mathbb{R}^d \rightarrow \mathbb{R}^d$ is a single-valued maximally monotone operator. The operator $B$ has a stochastic oracle $B_{\xi}$ that is unbiased and Lipschitz in mean. Under some conditions, they have proved the almost sure convergence of the proposed algorithm.

Motivated by the aforementioned works, we consider the following problem setting. Let  $\mathcal{X}$ be a separable Hilbert space, and suppose that the monotone operator $B$  in problem \eqref{ABC} admits a finite-sum representation, i.e., $B= \sum_{i=1}^N B_i$, where each $B_i$ is $L_i$-Lipschitz. Specifically, we study the problem:
\begin{equation}\label{PROB2}
\mbox{find} \ x \in \mathcal{X}  \ \ \mbox{such that} \ 0\in Ax+\sum_{i=1}^N B_ix+Cx.
\end{equation}
The assumptions on the other operators remain the same as those in problem \eqref{ABC}. Notably, extending stochastic Tseng's method with variance-reduction to solve problem \eqref{PROB2} introduces an additional requirement: the cocoercive operator $C$ must also possess a finite-sum structure.

\vskip 5mm

\subsection{Contributions}
The FBHF splitting algorithm is the first splitting algorithm proposed for  solving problem \eqref{ABC}. It integrates the classical forward-backward splitting algorithm with Tseng's forward-backward-forward splitting scheme. Due to its flexibility and efficiency, it has been extended and generalized to various problem settings.
In this work,  we propose a nonlinear forward-backward-half forward splitting with momentum (Algorithm \ref{algorithm1}) to solve the monotone inclusion problem \eqref{ABC}. Unlike existing approach \cite{FBHF} that require operators $M_k$ to be the bounded linear operators, $M_k$ in this paper are nonlinear operators. The convergence analysis in this paper imposes distinct assumptions on $M_k$ compared to \cite{FBHF}, we emphasize that these assumptions are neither stronger nor weaker but provide an alternative theoretical framework with a tailored proof procedure. Furthermore, we develop a stochastic variance-reduced forward-backward-half forward splitting algorithm with momentum (Algorithm \ref{algorithm5}) for the finite-sum structured problem \eqref{PROB2} (a special case of \eqref{ABC}). Under appropriate parameter settings, Algorithm \ref{algorithm5} reduces to the deterministic Algorithm \ref{algorithm1}.
The key contributions of this paper can be summarized as follows:
\begin{itemize}
	\item[(i)] We propose a nonlinear forward-backward-half forward splitting with momentum for solving problem \eqref{ABC}. Furthermore, weak convergence is established under the well-defined parameter condition. Under strong monotonicity conditions, we establish that the Algorithm \ref{algorithm1} achieves linear convergence rate.
	\item[(ii)] Based on the characteristics of the proposed momentum-based algorithm, we not only recover several existing algorithms but also develop a new class of operator splitting algorithms for solving monotone inclusions involving the sum of four operators. This approach effectively leverages the structural properties of each operator, ensuring good convergence behavior and computational efficiency.
	\item[(iii)] We propose a stochastic variance-reduced forward-backward-half forward splitting algorithm with momentum to solve problem \eqref{PROB2}, where the problem possesses a finite-sum structure. In a separable real Hilbert space, the weak almost sure convergence  of Algorithm \ref{algorithm5} is obtained. And the linear convergence rate of Algorithm \ref{algorithm5} is also given under the stronger assumption, i.e. strong monotonicity.
\end{itemize}
The paper is organized as follows. Section 2 recalls some fundamental definitions and lemmas essential for subsequent analysis. In Section 3, we introduce a nonlinear forward-backward-half forward splitting algorithm with momentum for solving problem \eqref{ABC}, establishing its weak convergence under standard assumptions and proving linear convergence under additional conditions of strong monotonicity. Section 4 further develops a stochastic variance-reduced variant with momentum to address problem \eqref{PROB2}, where we demonstrate the weak almost sure convergence and establish the linear convergence rates. Section 5 presents a series of numerical experiments conducted on both synthetic datasets and real-world quadratic programming problems arising in portfolio optimization. Finally, we summarize the main conclusions.

\section{ Preliminaries}
	
Throughout this paper,  $\mathcal{H}$ and $\mathcal{G}$ denote real Hilbert spaces equipped with inner product  $\langle \cdot, \cdot\rangle$ and the corresponding norm $\|\cdot\|$. Let $\mathcal{X}$ be a separable Hilbert space,  and let $\mathcal{B}$ denote the Borel $\sigma$-algebra on $\mathcal{X}$. Let $\mathbb{N}$ denote the set of nonnegative integers and $\mathbb{N}_0$ the set of positive integers. Let $(\Omega, \mathcal{F}, P)$ be a probability space. A $\mathcal{X}$-valued random variable is a measurable mapping $x: (\Omega, \mathcal{F})\rightarrow (\mathcal{X}, \mathcal{B})$. The $\sigma$-algebra
generated by a family $\Phi$ of random variables is denoted by $\sigma(\Phi)$. Let $\mathscr{F}= \{\mathcal{F}_k\}_{k\in \mathbb{N}}$
be a sequence of sub-sigma algebras of $\mathcal{F}$ such that $\mathcal{F}_k \subset \mathcal{F}_{k+1}$.  The probability
mass function $P_{\xi}(\cdot)$ is supported on the finite set $\{1, \ldots, N\}$.  Strong  and weak convergence are denoted by $``\rightarrow"$ and $``\rightharpoonup"$, respectively.  For a linear operator $L: \mathcal{H} \rightarrow \mathcal{G}$, we denote its adjoint by $L^*: \mathcal{G} \rightarrow \mathcal{H}$, and its operator norm by $\|L\|$. An operator $S:\mathcal{H} \rightarrow \mathcal{H}$ is said to be self-adjoint if $S^* = S$, and strongly monotone if there exists a constant $m \in (0, +\infty)$ such that
\begin{equation}
\langle Sx | x \rangle \geq m \|x\|^2, \quad \forall x \in \mathcal{H}.
\end{equation}
Define $$\mathcal{P}(\mathcal{H})=\{S:\mathcal{H} \rightarrow \mathcal{H}|S \ \hbox{is linear, self-adjoint, and strongly positive}\}.$$ Then for any $S \in \mathcal{P}(\mathcal{H})$, the operator $S$ is invertible and its inverse $S^{-1}$ also belongs to $\mathcal{P}(\mathcal{H})$. The inner product and norm induced by $S$ are denoted by $\langle \cdot ,\cdot \rangle_S = \langle S (\cdot), \cdot \rangle$ and $\|\cdot\|_S$, respectively. These norms are equivalent on the Hilbert space $\mathcal{H}$; that is, there exist $C_1, \ C_2>0$ such that $C_1\|x\| \geq \|x\|_S \geq C_2\|x\|$ for all $x\in \mathcal{H}$.

Observe that $S =S^{\frac{1}{2}}\circ S^{\frac{1}{2}},$ where $S^{\frac{1}{2}}$ is a self-adjoint, strongly monotone linear operator. As a result, the Cauchy-Schwarz inequality extends to the norms
induced by $S$ and $S^{-1}$  in the following form:
\begin{equation}
|\langle x,u \rangle| = |\langle S^{-\frac{1}{2}}x,S^{\frac{1}{2}}u \rangle|\leq  \| S^{-\frac{1}{2}}x\| \|S^{\frac{1}{2}}u \| =
\|x\|_{S^{-1}}\|u\|_S, \quad \forall (x,u) \in \mathcal{H}\times\mathcal{H}. \end{equation}
Moreover, the following identity holds for all $x,y,z\in \mathcal{H}$:
\begin{equation}
2\langle x - y | y - z\rangle_S = \|x - z\|_S^2 - \|x - y\|_S^2 - \|y - z\|_S^2.
\end{equation}

\begin{definition}{\rm(\cite{BC2011}\ )}
			{\rm
				Let $A: \mathcal{H} \rightarrow 2^{\mathcal{H}}$  be a set-valued mapping and the graph of $A$ is defined by $\gra(A)=\{(x,u) \in \mathcal{H}\times\mathcal{H}:u \in Ax\}.$ $A$ is said to be
				\begin{itemize}
					\item[(i)]  monotone if $\langle u-v,x-y \rangle \geq 0$ for all $(x,u),(y,v) \in \gra(A)$.
					\item[(ii)]  maximally monotone if there exists no monotone operator $B: \mathcal{H} \rightarrow 2^{\mathcal{H}}$ such that gra$(B)$ properly contains gra$(A),$ i.e., for every $(x,u) \in \mathcal{H}\times\mathcal{H}$
					$$
					(x,u) \in \gra (A)  \ \ \Leftrightarrow \ \ \langle u-v , x-y \rangle \geq 0,  \ \  \forall(y,v)\in \gra (A).
					$$
				\end{itemize}
			}
		\end{definition}
		\begin{definition}
			{\rm
				Let $T: \mathcal{H} \rightarrow {\mathcal{H}}$ be a single-valued operator, $T$ is said to be
				\begin{itemize}
					\item[(i)]  $L$-Lipschitz continuous w.r.t. $S$, if there exists a constant $ L > 0$, such that
					$$
					\|Tx-Ty\|_{S^{-1}}\leq  L\|x-y\|_S,\quad\forall x, y\in \mathcal{H}. \\
					$$
					\item[(ii)]  $\beta^{-1}$-cocoercive operator w.r.t. $S$, if there exists a constant $ \beta > 0$, such that
					$$
					\langle Tx-Ty,x-y \rangle \geq  \beta^{-1}\|Tx-Ty\|_{S^{-1}}^2,\quad\forall x, y\in \mathcal{H}.
					$$

An operator is $\beta^{-1}$-cocoercive w.r.t. $S$ is also $\beta$-Lipschitz continuous w.r.t. $S$.
				\end{itemize}
			}
		\end{definition}

\begin{definition}{\rm(\cite{BC2011}\ )}	{\rm
We denote the proximity of operator $f$ by
\begin{equation}
	\prox_{\gamma f}(x) :x \mapsto \argmin_{y \in \mathcal{H}} \left(f(y)+\frac{1}{2\gamma} \|x-y\|^2\right),
\end{equation}
where $\gamma$ is the step-size parameter.}
\end{definition}

We shall make full use of the following lemmas to conduct the convergence analysis.

\begin{lemma}{\rm(\cite{{Morin}}\ )}\label{lem0}
Let $C$ be the $\beta^{-1}$-cocoercive operator w.r.t. $S$. Then the following inequality holds:
\begin{equation}
\langle Cx - Cy, z - y\rangle\geq-\frac{\beta}{4}\|z-x\|_{S}^{2},\quad\forall x,y,z\in \mathcal{H}.
\end{equation}
\end{lemma}
	
\begin{lemma}{\rm(\cite{BC2011}\ )}\label{lem3}
Let $A: \mathcal{H} \rightarrow 2^{\mathcal{H}}$ be maximally monotone. Then $\gra(A)$ is sequentially closed in $\mathcal{H}^{weak}\times\mathcal{H}^{strong}$, i.e., for every sequence $(x_{k},y_{k})_{k \in \mathbb{N}}$ in $\gra(A)$ and $(x,y)\in \mathcal{H}\times\mathcal{H}$, if $x_{k}\rightharpoonup x$ and $y_{k}\rightarrow y$, then $(x,y) \in \gra(A)$. \vskip 1mm
\end{lemma}
\begin{lemma}\label{lem1}{\rm(\cite{BC2011}\ )}
Let $D$ be a nonempty set of $\mathcal{H}$, and $\{x_{k}\}_{k \in \mathbb{N}}$ be a sequence in $\mathcal{H}$ such that the following two conditions hold:\vskip 1mm
\begin{itemize}
\item[{\rm(i)}]  For every $x \in D$, $\lim\limits_{k\rightarrow \infty}\|x_{k}-x\|$ exists;
\noindent
\item[{\rm(ii)}] Every weak sequential cluster point of $\{x_{k}\}_{k \in \mathbb{N}}$ belongs to $D$.
\end{itemize}
Then the sequence $\{x_{k}\}_{k \in \mathbb{N}}$ converges weakly to a point in $D$.
\end{lemma}

\begin{lemma}{\rm(\cite{Comettes}\ )} \label{Comettes}	
Let $F$ be a nonempty closed subset of a separable real Hilbert space $\mathcal{K}$ and  $\phi:[0,+\infty) \rightarrow [0,+\infty)$ be a strictly increasing function such that $\lim_{t \rightarrow +\infty}\phi(t)=+\infty$. Let $\{x^k\}_{k \in \mathbb{N}}$ be a sequence of $\mathcal{K}$-valued random variables and  $\chi_k=\sigma(x^0,\cdots,x^k), \forall k \in \mathbb{N}.$  Suppose that, for every $z \in F$, there exist $\{\beta_k(z)\}_{k\in\mathbb{N}}$, $\{\xi_k(z)\}_{k\in\mathbb{N}}$, and $\{\zeta_k(z)\}_{k\in\mathbb{N}}$ be nonnegative $\chi_k$-measurable random variables such that $\sum_{k=0}^{\infty} \beta_k(z) < \infty$, $\sum_{k=0}^{\infty} \xi_k(z) < \infty$ and
\begin{equation*}
\mathbb{E}(\phi(\|x^{k+1}-z\|) |  \chi_k) \leq (1+\beta_k(z))\phi(\|x^k-z\|)+\xi_k(z)-\zeta_k(z),\quad\forall k\in\mathbb{N}.
\end{equation*}
Then the following hold:\vskip 1mm
\noindent
{\rm (i)}  $(\forall z \in F),$ $\sum_{k=0}^{\infty} \zeta_k(z) < \infty$ almost surely. \vskip 1mm
\noindent
{\rm (ii)} There exists $\Xi \in \mathcal{F}$ such that $P(\Xi)=1$, for every $\theta \in \Xi$ and every $z \in F$, $\{\|x^k(\theta)-z\|\}_{k\in \mathbb{N}}$ converges.  \vskip 1mm
\noindent
{\rm (iii)} Suppose that all weak cluster points of $\{x^k\}_{k \in \mathbb{N}}$ belong to $F$ almost surely, then $\{x^k\}_{k \in \mathbb{N}}$ converges weakly almost surely to an $F$-valued random variable.
\end{lemma}

\section{Nonlinear forward-backward-half forward with momentum}	
In this section,  we propose an algorithmic framework for solving the monotone inclusion problem \eqref{ABC} via a nonlinear forward-backward-half forward splitting algorithm with momentum. We begin by stating the following assumptions:
			\begin{assumption}
				\label{assumption1}
				
				Assume the following conditions hold:
				\begin{itemize}
					\item[{\rm(i)}]  $A:\mathcal{H} \rightarrow 2^{\mathcal{H}}$ is  maximally monotone.
					
					\item[{\rm(ii)}]   $B:\mathcal{H} \rightarrow \mathcal{H}$ is a single-valued monotone and $\mu$-Lipschitz continuous w.r.t. $S$, where $S \in \mathcal{P}(\mathcal{H})$.

					\item[{\rm(iii)}]  $C:\mathcal{H}\rightarrow \mathcal{H}$ is $\beta^{-1}$-cocoercive w.r.t. $S$, where $S \in \mathcal{P}(\mathcal{H})$, for some $\beta>0$.
					
					\item[{\rm(iv)}]   The solution set of problem \eqref{ABC}, defined by
					$$
					\zer(A+B+C):=\{x \in \mathcal{H}:0 \in Ax+Bx+Cx\}
					$$
					is nonempty.
                  \item[{\rm(v)}] The nonlinear kernel $M_k : \mathcal{H} \rightarrow \mathcal{H}$ satisfies the condition that  $\gamma_k M_k - S$ is $L_k$-Lipschitz continuous w.r.t. $S$, for some $L_k \in [0,1)$, $\gamma_k\geq\gamma,$ $\forall k \in \mathbb{N}$.
				\end{itemize}
			\end{assumption}

        The following result is taken from \cite{Morin}.
\begin{proposition}{\rm(\cite{{Morin}}\ )}
Suppose Assumption \ref{assumption1} (v) holds. Then, for each  $k \in \mathbb{N}$, the operator $M_k$ is $2\gamma^{-1}$-Lipschitz continuous w.r.t. $S$, maximally monotone, and strongly monotone w.r.t. $S$.
\end{proposition}		
				
\begin{algorithm}
\label{algorithm1}
				\hrule
				\vskip 1mm
				\noindent\textbf{\footnotesize{Nonlinear forward-backward-half forward with momentum}}
				\vskip 1mm
				\hrule
				
				\vskip 2mm
\noindent Let $S \in \mathcal{P}(\mathcal{H}),$ $M_k: \mathcal{H} \rightarrow \mathcal{H},$ $\gamma_k > 0 \ (\forall k \in \mathbb{N})$ and initial point $x_{0}, \ u_{0} \in \mathcal{H}$. For $k=0,1,2,\cdots$, do
\begin{equation}
\label{algorithm_1}
\aligned
&y_k=(M_k+A)^{-1}(M_kx_k-(B+C)x_k+\gamma_k^{-1}u_k), \\
&x_{k+1}=y_k-\gamma_kS^{-1}By_k+\gamma_kS^{-1}Bx_k, \\
&u_{k+1}=(\gamma_kM_k-S)y_k-(\gamma_kM_k-S)x_k.
\endaligned
\end{equation}
\hrule
\hspace*{\fill}
\end{algorithm}

\subsection{Weak convergence}
 In this subsection, we present the convergence analysis for Algorithm \ref{algorithm1}. Let $\{x_k\}_{k \in \mathbb{N}}$ and $\{y_k\}_{k \in \mathbb{N}}$ be the sequences generated by Algorithm \ref{algorithm1}. For any $ x \in \zer(A+B+C)$, we define the function
\begin{equation}\label{Phik}
\Phi_{k}(x):=\|x_k-x\|_{S}^2+2\langle u_k, x_k-x\rangle +L_{k-1}\|y_{k-1}-x_{k-1}\|_S^2,
\end{equation}
which helps to establish the weak convergence of the Algorithm \ref{algorithm1}.		

			\begin{lemma}
				\label{lemma}
				Suppose that Assumption {\rm\ref{assumption1}} holds. Let $\{x_k\}_{k \in \mathbb{N}}$, $\{y_k\}_{k \in \mathbb{N}}$, and $\{u_k\}_{k \in \mathbb{N}}$ be the sequences generated by Algorithm \ref{algorithm1}. Then, for any $x^* \in \zer(A+B+C)$ and all $k \in \mathbb{N}_0$, the following inequality holds:
\begin{equation}
\Phi_{k+1}(x^*) \leq \Phi_{k}(x^*)-(1- L_{k-1}-L_k-2\gamma_kL_k\mu-\gamma_k^2\mu^2-\frac{\gamma_k\beta}{2})\|y_k-x_k\|_S^2. \label{1}
\end{equation}
			\end{lemma}
			
			\begin{proof}
From Algorithm \ref{algorithm1}, we have
$$
y_k=(M_k+A)^{-1}(M_kx_k-(B+C)x_k+\gamma_k^{-1}u_k),
$$
which is equivalent to the inclusion
				\begin{equation}
					M_kx_k-M_ky_k-(B+C)x_k+\gamma_k^{-1}u_k \in Ay_k.     \label{2}
				\end{equation}
				Since $x^* \in \zer(A+B+C)$, we have
				\begin{equation}
					-(B+C)x^* \in Ax^*.                                 \label{3}
				\end{equation}
				By combining \eqref{2} and \eqref{3}, and using the monotonicity of $A$, we obtain				
                \begin{equation}
				0 \leq \langle M_kx_k-M_ky_k-(B+C)x_k+\gamma_k^{-1}u_k+(B+C)x^* , y_k-x^* \rangle. 					\label{4}
				\end{equation}
Multiplying both sides of \eqref{4} by $2\gamma_k$ and using the definition of $u_{k+1}$, we have
\begin{equation}
\aligned
0& \leq 2\langle Sx_k-Sy_k+u_k-u_{k+1}-\gamma_k(B+C)x_k+\gamma_k(B+C)x^* , y_k-x^* \rangle\\
& = 2\langle Sx_k-Sy_k, y_k-x^* \rangle +2\langle u_k-u_{k+1},y_k-x^* \rangle-2\gamma_k \langle Bx_k-Bx^*,y_k-x^*\rangle \\
& \quad -2\gamma_k\langle Cx_k-Cx^*,y_k-x^*\rangle. 	
\endaligned
\label{5}
\end{equation}
We now analyze the four inner products appearing on the right-hand side of inequality \eqref{5}. Recalling the definition of $x_{k+1}$, we have
				\begin{equation}
					\begin{aligned}
					   & 2\langle Sx_k -Sy_k, y_k-x^*\rangle \\
                       =&\|x_k-x^*\|_S^2-\|y_k-x^*\|_S^2-\|y_k-x_k\|_S^2 \\
                       =&\|x_k-x^*\|_S^2-\|x_{k+1}+\gamma_kS^{-1}By_k-\gamma_kS^{-1}Bx_k-x^*\|_S^2-\|y_k-x_k\|_S^2 \\
                       =&\|x_k-x^*\|_S^2-\|x_{k+1}-x^*\|_S^2-2\gamma_k\langle x_{k+1}-x^*,By_k-Bx_k\rangle \\
                       &-\gamma_k^2\|By_k-Bx_k\|_{S^{-1}}^2-\|y_k-x_k\|_S^2, \\
					\end{aligned}  \label{7}
				\end{equation}
				and
				\begin{equation}
					\begin{aligned}
					   & 2\langle u_k-u_{k+1}, y_k-x^*\rangle \\
                       =& 2\langle u_k,x_k-x^*\rangle+2\langle u_k,y_k-x_k\rangle-2\langle u_{k+1},x_{k+1}-x^*\rangle \\
                       & -2\gamma_k\langle u_{k+1},S^{-1}By_k-S^{-1}Bx_k\rangle. \\
                       \leq &2\langle u_k,x_k-x^*\rangle+2\|u_k\|_{S^{-1}}\|y_k-x_k\|_S-2\langle u_{k+1},x_{k+1}-x^*\rangle \\
                       & +2\gamma_k\|u_{k+1}\|_{S^{-1}}\|By_k-Bx_k\|_{S^{-1}} \\
                       \leq&  2\langle u_k,x_k-x^*\rangle-2\langle u_{k+1},x_{k+1}-x^*\rangle+L_{k-1}\|y_{k-1}-x_{k-1}\|_S^2 \\
                       & +L_{k-1}\|y_{k}-x_{k}\|_S^2+2\gamma_kL_k\mu\|y_k-x_k\|_S^2. \\
					\end{aligned}  \label{8}
				\end{equation}
By Lemma \ref{lem0}, we get
				\begin{equation}
					-2\gamma_k\langle Cx_k-Cx^* , y_k-x^* \rangle \leq \frac{\gamma_k\beta}{2}\|y_k-x_k\|_S^2. \label{6}
				\end{equation}
				Combining \eqref{7}, \eqref{8},  and \eqref{6}, we deduce that
				\begin{equation}
					\begin{aligned}
						& \quad\|x_{k+1}-x^*\|_S^2+2\langle u_{k+1}, x_{k+1}-x^*\rangle\\
						& \leq \|x_k-x^*\|_S^2 +2\langle u_k,x_k-x^*\rangle+L_{k-1}\|y_{k-1}-x_{k-1}\|_S^2 \\
                        & \quad-2\gamma_k \langle Bx_k-Bx^*,y_k-x^*\rangle-2\gamma_k\langle x_{k+1}-x^*,By_k-Bx_k\rangle-\gamma_k^2\|By_k-Bx_k\|_{S^{-1}}^2 \\
                        &  \quad-\|y_k-x_k\|_S^2+L_{k-1}\|y_{k}-x_{k}\|_S^2+2\gamma_kL_k\mu\|y_k-x_k\|_S^2+\frac{\gamma_k\beta}{2}\|y_k-x_k\|_S^2\\
                        & = \|x_k-x^*\|_S^2 +2\langle u_k,x_k-x^*\rangle+L_{k-1}\|y_{k-1}-x_{k-1}\|_S^2 \\
                        & \quad-2\gamma_k \langle Bx_k-Bx^*,y_k-x^*\rangle-2\gamma_k\langle y_k-\gamma_kS^{-1}By_k+\gamma_kS^{-1}Bx_k-x^*,By_k-Bx_k\rangle \\
                        & \quad-\gamma_k^2\|By_k-Bx_k\|_{S^{-1}}^2 -(1-L_{k-1}-2\gamma_kL_k\mu-\frac{\gamma_k\beta}{2})\|x_k-y_k\|_S^2 \\
                        & =\|x_k-x^*\|_S^2 +2\langle u_k,x_k-x^*\rangle+L_{k-1}\|y_{k-1}-x_{k-1}\|_S^2 \\
                        & \quad -2\gamma_k \langle By_k-Bx^*,y_k-x^*\rangle+\gamma_k^2\|By_k-Bx_k\|_{S^{-1}}^2\\
                        & \quad -(1-L_{k-1}-2\gamma_kL_k\mu-\frac{\gamma_k\beta}{2})\|x_k-y_k\|_S^2.
					\end{aligned} 					\label{9}
				\end{equation}
Applying the monotonicity of $B$ and rearranging the terms yields
\begin{equation}
\begin{aligned}
& \quad\|x_{k+1}-x^*\|_S^2+2\langle u_{k+1}, x_{k+1}-x^*\rangle+L_{k}\|y_{k}-x_{k}\|_S^2 \\
& \leq \|x_k-x^*\|_S^2 +2\langle u_k,x_k-x^*\rangle+L_{k-1}\|y_{k-1}-x_{k-1}\|_S^2 \\
& \quad -(1-L_{k-1}-L_k-2\gamma_kL_k\mu-\gamma_k^2\mu^2-\frac{\gamma_k\beta}{2})\|x_k-y_k\|_S^2.
\end{aligned} 					\label{10}
\end{equation}
Then \eqref{1} follows from the definition of $\Phi_k(x)$ given in \eqref{Phik}.
			\end{proof}
			
			\begin{theorem}
				\label{theorem1}
				{
					\noindent
Suppose that Assumption {\rm\ref{assumption1}} holds, and there exists a constant $\epsilon >0$ such that
\begin{equation}
1-L_{k-1}-L_k-2\gamma_kL_k\mu-\gamma_k^2\mu^2-\frac{\gamma_k\beta}{2} \geq \epsilon,	\label{11}
\end{equation}
holds for all $k \in \mathbb{N}_0$. Then the sequence $\{x_k\}_{k \in \mathbb{N}}$ generated by Algorithm \ref{algorithm1} converges weakly to a point in $\zer(A+B+C)$.
				}
			\end{theorem}
			
\begin{proof}
Using the definition of $\Phi_k(x)$ in \eqref{Phik}, the Lipschitz continuity of $\gamma_kM_k-S$ and the Cauchy-Schwarz inequality, we deduce that
\begin{equation}
\begin{aligned}
\Phi_{k}(x)&=\|x_k-x\|_{S}^2+2\langle u_k, x_k-x\rangle +L_{k-1}\|y_{k-1}-x_{k-1}\|_S^2 \\
& \geq \|x_k-x\|_{S}^2-L_{k-1}(\|y_{k-1}-x_{k-1}\|_S^2+\|x_k-x\|_S^2)+L_{k-1}\|y_{k-1}-x_{k-1}\|_S^2 \\
& = (1-L_{k-1})\|x_k-x\|_S^2\geq 0,
\end{aligned}  \label{12}
\end{equation}
which shows that the sequence $\{\Phi_{k}(x)\}_{k \in \mathbb{N}_0}$ is nonnegative.  Since it is also nonincreasing, it is convergent. Let $N \in \mathbb{N}$. Summing inequality \eqref{1} for $k = 1$ to $N$ yields
\begin{equation}
\begin{aligned}
\sum_{k=1}^N(1- L_{k-1}-L_k-2\gamma_kL_k\mu-\gamma_k^2\mu^2-\frac{\gamma_k\beta}{2})\|y_k-x_k\|_S^2 &\leq \Phi_{1}(x)- \Phi_{N+1}(x)\\
& \leq \Phi_{1}(x) < + \infty.
\end{aligned}  \label{13}
\end{equation}
Thanks to condition \eqref{11}, we obtain that $\lim\limits_{k\rightarrow \infty}\|y_k-x_k\|_S=0$. Therefore,
\begin{equation}
\lim_{k\rightarrow \infty}\Phi_k(x)=\lim_{k\rightarrow \infty}\|x_k-x\|_S^2,  \label{14}
\end{equation}
which implies that the sequence $\{x_k\}_{k \in \mathbb{N}}$ is bounded. The sequence $\{x_k\}_{k \in \mathbb{N}}$ admits at least one weakly convergent subsequence $\{x_{k_n}\}_{n \in \mathbb{N}}$. Without loss of generality, assume that $x_{k_n} \rightharpoonup \bar{x} \in \mathcal{H}$ as $n\rightarrow \infty$.
Define
\begin{equation}
\label{delta}
\Delta_k := M_kx_k-M_ky_k-(B+C)(x_k-y_k)+\gamma_k^{-1}u_k,
\end{equation}
Then, by \eqref{2}, we have $(y_k , \Delta_k) \in \gra(A+B+C)$ for all $k \in \mathbb{N}$. Since $M_k$ is $2\gamma^{-1}$-Lipschitz continuous w.r.t. $S$, $B$ is $\mu$-Lipschitz continuous w.r.t. $S$, and $C$ is $\beta^{-1}$-cocoercive w.r.t. $S$, if follows that
\begin{equation}
\aligned \| \Delta_k \|_{S^{-1}}
& \leq \|M_kx_k-M_ky_k\|_{S^{-1}}+\|(B+C)(x_k-y_k)\|_{S^{-1}}+\frac{1}{\gamma_k}\|u_k\|_{S^{-1}}  \\
& \leq \frac{2}{\gamma}\|x_k-y_k\|_S+(\mu+\beta)\|x_k-y_k\|_S+\frac{L_{k-1}}{\gamma_k}\|x_{k-1}-y_{k-1}\|_S.
\endaligned \label{16}
\end{equation}
From the estimates above, we concludate that $\Delta_{k_n}\rightarrow 0$. Furthermore, since $B$ has full domain, it follows from Corollary 25.5 (i) in \cite{BC2011} that the sum $A+B$ is maximally monotone. Moeover, by Lemma 2.1 in \cite{Showalter} and the assumption that $C$ is cocoercive w.r.t. $S$, it follows that $A+B+C$ is maximally monotone. Then, by Lemma \ref{lem3}, we have $(\bar{x} , 0) \in \gra(A+B+C)$, which implies that $\bar{x} \in \zer(A+B+C)$. Finally, applying Lemma \ref{lem1}, we conclude that  the sequence $\{x_k\}_{k \in \mathbb{N}}$ converges weakly to a point in $\zer(A+B+C)$.
\end{proof}

\begin{remark}
\rm We present four special cases of Algorithm \ref{algorithm1}.
\begin{itemize}
\item[{\rm (i)}] When $B=0$, we have $\mu=0$. In this case, Algorithm \ref{algorithm1} reduces to the nonlinear forward-backward splitting with momentum studied in \cite{Morin}, and the condition on $\gamma_k$ simplifies to $1-L_{k-1}-L_k-\frac{\gamma_k\beta}{2} \geq \epsilon$, which coincides with the step-size condition proposed therein.

\item[{\rm (ii)}] When $S=\Id$ and $M_k=\frac{\Id}{\gamma_k}$, we have $L_k=0$ and $u_k=0$.  In this case, Algorithm \ref{algorithm1} simplifies to the FBHF splitting algorithm \eqref{FBHF} with $X=\mathcal{H}$. Moreover, the constraint in \eqref{11} reduces to $2\gamma_{k}^{2}\mu^{2}+\gamma_{k}\beta \leq 2-\epsilon$, which is consistent with the upper bound of $\gamma_k$ given in \cite{FBHF}.

\item[{\rm (iii)}] When $C=0$, $S=\Id$, and $M_k=\frac{\Id}{\gamma_k}$, we again have $L_k=0$, $u_k=0$, and $\beta=0$. Then Algorithm \ref{algorithm1} reduces to the forward-backward-forward splitting algorithm introduced by Tseng \cite{Tseng}, and the step-size condition on $\gamma_k$ is recovered accordingly.

\item[{\rm (iv)}] Suppose $A=A_1+A_2$, where $A_1:\mathcal{H} \rightarrow 2^{\mathcal{H}}$ is maximally monotone, $A_2:\mathcal{H} \rightarrow 2^{\mathcal{H}}$ is $L$-Lipschitz continuous, and $A_1+A_2$ is maximally monotone. Further, assume that $S=\Id$, $\gamma_k\equiv \gamma$, and $M_k=\frac{\Id}{\gamma}-A_2$. Then Algorithm \ref{algorithm1} becomes
\begin{equation}
\label{4operator}
\left\{
\begin{array}{lr}
y_k=J_{\gamma A_1}(x_k-\gamma A_2x_k-\gamma Bx_k-\gamma Cx_k-\gamma( A_2y_{k-1}-A_2x_{k-1})), & \\
x_{k+1}=y_k-\gamma By_k+\gamma Bx_k. & \\
\end{array}
\right.
\end{equation}
In this case, $\gamma M_k-\Id$ is $\gamma L$-Lipschitz continuous. As a result, the condition  in \eqref{11} becomes
\begin{equation*}
1-2\gamma L-2\gamma^2L\mu-\gamma^2\mu^2-\frac{\gamma \beta}{2}  \geq \epsilon.
\end{equation*}
Algorithm \eqref{4operator} can thus be interpreted as a four-operator splitting scheme for finding a zero of the sum of a maximally monotone operator, two Lipschitz continuous operators, and a cocoercive operator.
\end{itemize}
\end{remark}

\subsection{Linear convergence}
In this subsection, we establish the $R$-linear convergence of the sequence $\{x_k\}_{k \in \mathbb{N}}$ generated by Algorithm \ref{algorithm1}, under the assumption that the operator $A$ is  strongly monotone.  An operator $A$ is said to be $\rho$-strongly monotone w.r.t. $S$ , for some $\rho >0$, if
\begin{equation}
 \langle u-v , x-y \rangle \geq \rho\|x-y\|_S^2, \quad  \forall(x,u)\in \gra (A) ,   \forall(y,v) \in \gra (A).
\end{equation}
If, instead, $B$ is $\rho$-strongly monotone, then one may equivalently shift the operator sum by a multiple of the identity operator, as follows
 $A+B=(A+\rho \Id)+(B- \rho \Id)$, without affecting maximal monotonicity. Therefore, strong monotonicity $B$ also implies $R$-linear convergence of the sequence, and the convergence proof proceeds analogously.
\begin{theorem}
\label{linear1}
{
\noindent
Let Assumption {\rm\ref{assumption1}} hold, and suppose that the operator $A$ is $\rho$-strongly monotone. Assume that there exist constants $\varepsilon_1, \varepsilon_2 >0$ such that $\gamma_k <\frac{1}{\varepsilon_1}$, and define
$$t=\min\left\{\frac{2\gamma_k m(1-\gamma_k \varepsilon_1)}{1+\frac{L_k}{ \varepsilon_2}}, \frac{\kappa+\nu}{L_k(\varepsilon_2+1)}\right\},$$
where $\kappa=1-L_{k-1}-L_k-2\gamma_kL_k\mu-\gamma_k^2\mu^2-\frac{\gamma_k\beta}{2}, \ \nu=2\gamma_k^2\rho\mu^2(\gamma_k-\frac{1}{\varepsilon_1})$.
Then the sequence $\{x_k\}_{k \in \mathbb{N}}$ generated by Algorithm \ref{algorithm1} converges $R$-linearly to a point $x^*$ in $\zer(A+B+C)$.
}
\end{theorem}
\begin{proof}
By applying the strong monotonicity of $A$ in Lemma \ref{lemma}, \eqref{4} yields the following inequality
\begin{equation}
 \rho \|y_k-x^*\|_S^2 \leq \langle M_kx_k-M_ky_k-(B+C)x_k+\gamma_k^{-1}u_k+(B+C)x^* , y_k-x^* \rangle. 					
\end{equation}
And the following inequality is propagated through the proof of Lemma \ref{lemma},
\begin{equation}
\label{100}
\begin{aligned}
& \quad\|x_{k+1}-x^*\|_S^2+2\langle u_{k+1}, x_{k+1}-x^*\rangle+L_{k}\|y_{k}-x_{k}\|_S^2 +2\gamma_k\rho\|y_k-x^*\|_S^2 \\
& \leq \|x_k-x^*\|_S^2 +2\langle u_k,x_k-x^*\rangle+L_{k-1}\|y_{k-1}-x_{k-1}\|_S^2 \\
& \quad -(1-L_{k-1}-L_k-2\gamma_kL_k\mu-\gamma_k^2\mu^2-\frac{\gamma_k\beta}{2})\|x_k-y_k\|_S^2.
\end{aligned} 					
\end{equation}
With the help of the Young's inequality ($2 \langle a,b \rangle \leq \varepsilon\|a\|^2 + \frac{1}{\varepsilon} \|b\|^2$ for all $a , b\in \mathcal{H},$ and  $\varepsilon>0$), we get
\begin{equation}
\label{101}
\begin{aligned}
&2\gamma_k\rho\|y_k-x^*\|_S^2 \\
= & 2\gamma_k\rho\|x_{k+1}+\gamma_kS^{-1}By_k-\gamma_kS^{-1}Bx_k-x^*\|_S^2 \\
= & 2\gamma_k\rho(\|x_{k+1}-x^*\|_S^2+2\gamma_k\langle x_{k+1}-x^*,By_k-Bx_k\rangle+\gamma_k^2\|By_k-Bx_k\|_{S^{-1}}^2) \\
\geq &  2\gamma_k\rho(1-\gamma_k\varepsilon_1)\|x_{k+1}-x^*\|_S^2+2\gamma_k^2\rho\mu^2(\gamma_k-\frac{1}{\varepsilon_1})\|y_k-x_k\|_{S}^2.
\end{aligned} 					
\end{equation}
Set
$$a_k=2\langle u_k,x_k-x^*\rangle+L_{k-1}\|y_{k-1}-x_{k-1}\|_S^2, $$
$$b_k=(\kappa+\nu)\|x_k-y_k\|_S^2.$$
Combining  \eqref{100} and \eqref{101} yields the following inequality
\begin{equation}
(1+2\gamma_k\rho(1-\gamma_k\varepsilon_1))\|x_{k+1}-x^*\|_S^2+a_{k+1}+b_k \leq \|x_{k}-x^*\|_S^2+a_k. \label{102}
\end{equation}
Using Lipschitzness of $u_{k+1}$, we have
\begin{equation}
\label{103}
\begin{aligned}
a_{k+1}
& = 2\langle u_{k+1},x_{k+1}-x^*\rangle+L_{k}\|y_{k}-x_{k}\|_S^2 \\
& \leq  \varepsilon_2L_k\|y_{k}-x_{k}\|_S^2+\frac{L_k}{\varepsilon_2}\|x_{k+1}-x^*\|_S^2+L_{k}\|y_{k}-x_{k}\|_S^2\\
& = L_k(\varepsilon_2+1)\|y_{k}-x_{k}\|_S^2+\frac{L_k}{\varepsilon_2}\|x_{k+1}-x^*\|^2_S.
\end{aligned} 					
\end{equation}
Then by the assumption that
$$t=\min\left\{\frac{2\gamma_k \rho(1-\gamma_k \varepsilon_1)}{1+\frac{L_k}{ \varepsilon_2}}, \frac{\kappa+\nu}{L_k(\varepsilon_2+1)}\right\},$$
we get
\begin{equation}
\label{104}
ta_{k+1} \leq  (2\gamma_k\rho(1-\gamma_k\varepsilon_1)-t)\|x_{k+1}-x^*\|^2_S+b_k.
\end{equation}
Combining \eqref{102} and \eqref{104}, we obtain
\begin{equation}
\begin{aligned}
\label{105}
(1+t)(\|x_{k+1}-x^*\|^2_S+a_{k+1}) &\leq  (1+2\gamma_k\rho(1-\gamma_k\varepsilon_1))\|x_{k+1}-x^*\|_S^2+a_{k+1}+b_k \\
& \leq \|x_{k}-x^*\|_S^2+a_k.
\end{aligned} 	
\end{equation}
On the other hand, by \eqref{12}, we know that there exists $k_1>0$ such that
\begin{equation}
\label{106}
\|x_{k+1}-x^*\|^2_S+a_{k+1} \geq (1-L_{k})\|x_{k+1}-x\|_S^2\geq k_1\|x_{k+1}-x\|_S^2.
\end{equation}
Therefore,
\begin{equation}
\begin{aligned}
\label{107}
k_1\|x_{k+1}-x^*\|^2_S & \leq \|x_{k+1}-x\|_S^2+a_{k+1}\\
&\leq \frac{\|x_{k}-x^*\|_S^2+a_k}{1+t}\leq \ldots \leq \frac{\|x_{1}-x^*\|_S^2+a_1}{(1+t)^k},
\end{aligned}
\end{equation}
which shows that the sequence $\{x_k\}_{k \in \mathbb{N}}$ generated by Algorithm \ref{algorithm1} converges $R$-linearly to a point $x^*$ in $\zer(A+B+C)$.
\end{proof}

\section{Stochastic variance-reduced nonlinear forward-backward-half forward with momentum}	
In this section, we mainly consider the structured monotone inclusion problem \eqref{PROB2} in a separable real Hilbert space, which is recalled below
\begin{equation*}
\mbox{find} \ x \in \mathcal{X}  \ \ \mbox{such that} \ 0\in Ax+\sum_{i=1}^N B_ix+Cx,
\end{equation*}
where $A: \mathcal{X} \rightarrow 2^{\mathcal{X}}$ is a maximally monotone operator, $C: \mathcal{X} \rightarrow \mathcal{X}$ is a $\frac{1}{\beta}$-cocoercive operator, and each $B_i: \mathcal{H} \rightarrow \mathcal{H}$ is a maximally monotone and $L_i$-Lipschitz operator.

Throughout this section, we assume access to an unbiased stochastic oracle $B_{\xi}$ satisfying $B(x)=\mathbb{E}[B_{\xi}(x)]$. We consider using the stochastic oracle $B_{\xi}$ in place of $B$ to perform the half forward step in Algorithm \ref{algorithm1}. This substitution leverages  the computational efficiency of stochastic approximations, thereby reducing the per-iteration cost while preserving the algorithm's convergence properties.
\begin{algorithm}
\label{algorithm5}
				\hrule
				\noindent\textbf{\footnotesize{Stochastic variance-reduced nonlinear FBHF with momentum}}
				\hrule
				
				\vskip 1mm
\noindent Let $S \in \mathcal{P}(\mathcal{X}),$ $M_k: \mathcal{X} \rightarrow \mathcal{X} \ (\forall k \in \mathbb{N}),$ $\gamma> 0$ and initial point $x_{0}, \ u_{0} \in \mathcal{X}.$ Probability $p \in (0, 1]$, probability distribution $Q$, $\lambda \geq 0$. For $k=0,1,2,\cdots$, do \\
\begin{equation}
\label{algorithm-5}
\aligned
&\bar{x}_k=\lambda x_k+(1-\lambda)\omega_k \\
&y_k=(M_k+A)^{-1}(M_k\bar{x}_k-(B+C)\omega_k+\gamma^{-1}u_k), \\
&u_{k+1}=(\gamma M_k-S)y_k-(\gamma M_k-S)\bar{x}_k. \\
&\hbox{Draw an index} \ \xi_k \ \hbox{according to} \ Q \\
&x_{k+1}=y_k-\gamma S^{-1}B_{\xi_k}y_k+\gamma S^{-1}B_{\xi_k}\omega_k, \\
&\omega_{k+1}=
\begin{cases}
x_{k+1},&\hbox{with probability} \,\,p\\
\omega_k,&\hbox{with probability} \,\,1-p\
\end{cases}
\endaligned
\end{equation}
\vskip 1mm			
\hrule
\hspace*{\fill}
\end{algorithm}
\begin{remark}{\rm
We introduce two fundamental types of stochastic oracles:
\begin{itemize}
\item[(i)] Uniform sampling: Define $B_{\xi}(x)=NB_i(x)$, where the sampling probability satisfies $P_{\xi}(i)=\Pro \{\xi=i\}=\frac{1}{N}$. In this case, the Lipschitz constant of the stochastic oracle is given by $\theta=\sqrt{N\sum_{i=1}^N L_i^2}$.
\item[(ii)] Importance sampling: Define $B_{\xi}(x)=\frac{1}{P_{\xi}(i)}B_i(x)$, where the sampling probability is chosen as $ P_{\xi}(i)=\Pro\{\xi=i\}=\frac{L_i}{\sum_{j=1}^N L_j}$. Under this scheme,  the corresponding Lipschitz constant becomes $\theta=\sum_{i=1}^N L_i$.
\end{itemize}}
\end{remark}

\subsection{The weak almost sure convergence}
In this subsection, we establish the weak almost sure convergence of Algorithm \ref{algorithm5} in a separable real Hilbert space. To facilitate the analysis, we introduce the following notation for conditional expectations: $$\mathbb{E}_k[\cdot]=\mathbb{E}[\cdot|\sigma(\xi_0,...,\xi_{k-1},w^{k})]$$
and
$$\mathbb{E}_{k+\frac{1}{2}}[\cdot]=\mathbb{E}[\cdot|\sigma(\xi_0,...,\xi_{k},w^{k})].$$

\begin{assumption} {\rm
	\label{assumption2}
	\begin{itemize}
		\item[(i)] The operator $A: \mathcal{X}\to 2^{\mathcal{X}}$ is maximally monotone;
        \item[(ii)] The operator $B$ admits an unbiased stochastic oracle $B_\xi$ such that $B(x)=\mathbb{E}[B_{\xi}(x)]$, and is $\theta$-Lipschitz in expectation w.r.t. $S$, for some $\theta >0$; that is
        \begin{equation}
        \label{lip}
        \mathbb{E}[\|B_{\xi}(u)-B_{\xi}(v)\|_{S^{-1}}^2] \leq \theta^2\|u-v\|_S^2, \quad \forall u,v \in \mathcal{X};
        \end{equation}
		\item[(iii)] $C:\mathcal{H}\rightarrow \mathcal{H}$ is $\beta^{-1}$-cocoercive w.r.t. $S$, where $S \in \mathcal{P}(\mathcal{X}),$ for some $\beta>0$.
		\item[(iv)] The solution set  of the problem \eqref{PROB2}, denoted by $\Omega^*$, is nonempty.
        \item[{\rm(v)}] The nonlinear kernel $M_k : \mathcal{X} \rightarrow \mathcal{X}$ satisfies  $\gamma M_k - S$ is $L_k$-Lipschitz continuous w.r.t. $S$, for some $L_k \geq 0$, $\gamma>0,$ $\forall k \in \mathbb{N}$.
	\end{itemize}}
\end{assumption}

For the iterates $\{x_k\}_{k \in \mathbb{N}}$, $\{\omega_k\}_{k \in \mathbb{N}}$, $\{y_k\}_{k \in \mathbb{N}}$, and $\{\bar{x}_k\}_{k \in \mathbb{N}}$ generated by Algorithm \ref{algorithm5}, and $\forall x \in \mathcal{X}$, we define the following Lyapunov-type function:
\begin{equation}
\label{Gammak}
\Gamma_{k}(x):=\lambda \|x_k-x\|_S^2+\frac{1-\lambda}{p}\|\omega_k-x\|_S^2+2\langle u_k, x_k-x\rangle +L_{k-1}\|y_{k-1}-\bar{x}_{k-1}\|_S^2,
\end{equation}
which plays a pivotal role in establishing the  weak almost sure convergence of the proposed algorithm.

The following Lemma is similar to Lemma \ref{lemma}. To ensure the completeness of
this paper, we provide a detailed proof.

\begin{lemma}
\label{lemma2}
{
\noindent
Suppose that Assumption \ref{assumption2} holds. Let  $\lambda \geq 0$, and $p \in (0,1]$. Then for $\{x_k\}_{k \in \mathbb{N}},$  $\{u_k\}_{k \in \mathbb{N}},$  $\{\omega_k\}_{k \in \mathbb{N}}$ and  $\{u_k\}_{k \in \mathbb{N}}$  generated by Algorithm \ref{algorithm5} and any $x^* \in \Omega^*$,
\begin{equation}
\begin{aligned}
\label{lypunov}
\mathbb{E}_k[\Gamma_{k+1}(x^*)]
\leq & \Gamma_{k}(x^*)-(\lambda-L_{k-1}-\gamma L_k\theta-\lambda(\gamma L_k\theta+L_k))\|y_k-x_k\|_S^2 \\ &-(1-\lambda-\gamma^2\theta^2-\frac{\gamma\beta}{2}-(1-\lambda)(\gamma L_k\theta+L_k))\|y_k-\omega_k\|_S^2
\end{aligned} 					
\end{equation}
holds for all $k \in \mathbb{N}_0.$
}
\end{lemma}
\begin{proof}
By Algorithm \ref{algorithm5}, we know that
$$y_k=(M_k+A)^{-1}(M_k\bar{x}_k-(B+C)x_k+\gamma^{-1}u_k),$$
which is equivalent to the inclusion
\begin{equation}
M_k\bar{x}_k-M_ky_k-(B+C)\omega_k+\gamma^{-1}u_k \in Ay_k.     \label{20}
\end{equation}
It follows from $x^* \in \Omega^*$ that
\begin{equation}
-(B+C)x^* \in Ax^*.          \label{21}
\end{equation}
Combining \eqref{20}, \eqref{21}, and the monotonicity of $A$, we get 				
\begin{equation}
0 \leq \langle M_k\bar{x}_k-M_ky_k-(B+C)\omega_k+\gamma^{-1}u_k+(B+C)x^* , y_k-x^* \rangle. 				\label{22}
\end{equation}
Multiplying both sides by $2\gamma$ and using the definition of $u_{k+1},$ we obtain
\begin{equation}
\aligned
0& \leq 2\langle S\bar{x}_k-Sy_k+u_k-u_{k+1}-\gamma(B+C)\omega_k+\gamma(B+C)x^* , y_k-x^* \rangle\\
& = 2\langle S\bar{x}_k-Sy_k, y_k-x^* \rangle +2\langle u_k-u_{k+1},y_k-x^* \rangle-2\gamma \langle B\omega_k-Bx^*,y_k-x^*\rangle \\
& \quad -2\gamma\langle C\omega_k-Cx^*,y_k-x^*\rangle. 	
\endaligned
\label{23}
\end{equation}
Now, let us analyze the last four inner products. From the definition of $x_{k+1}$, we have
\begin{equation}
\begin{aligned}
& 2\langle S\bar{x}_k -Sy_k, y_k-x^*\rangle \\
=&2\langle \bar{x}_k -x_{k+1}-\gamma S^{-1}B_{\xi_k}y_k+\gamma S^{-1}B_{\xi_k}\omega_k , y_k-x^*\rangle_S \\
=&2\langle \bar{x}_k -x_{k+1}, y_k-x^*\rangle_S+2\gamma\langle B_{\xi_k}\omega_k-B_{\xi_k}y_k , y_k-x^*\rangle\\
=& 2\langle x_{k+1}-y_k, x^*-y_k \rangle + 2\langle y_k-\bar{x}_k, x^*-y_k \rangle +2\gamma\langle B_{\xi_k}\omega_k-B_{\xi_k}y_k , y_k-x^*\rangle\\
=& \|x_{k+1}-y_k\|^2+\|x^*-y_k\|^2-\|x_{k+1}-x^*\|^2+ 2\lambda\langle y_k-x_k , x^*-y_k \rangle \\
&+ 2(1-\lambda)\langle y_k-\omega_k , x^*-y_k \rangle  +2\gamma\langle B_{\xi_k}\omega_k-B_{\xi_k}y_k , y_k-x^*\rangle\\
=& \|x_{k+1}-y_k\|^2+\|x^*-y_k\|^2-\|x_{k+1}-x^*\|^2+ \lambda(\|x_k-x^*\|^2-\|y_k-x_k\|^2-\|y_k-x^*\|^2)\\
& +(1-\lambda)(\|\omega_k-x^*\|^2-\|y_k-\omega_k\|^2-\|y_k-x^*\|^2)+2\gamma\langle B_{\xi_k}\omega_k-B_{\xi_k}y_k , y_k-x^*\rangle\\
=&\|x_{k+1}-y_k\|^2-\|x_{k+1}-x^*\|^2+ \lambda\|x_k-x^*\|^2-\lambda\|y_k-x_k\|^2 +(1-\lambda)\|\omega_k-x^*\|^2 \\
&-(1-\lambda)\|y_k-\omega_k\|^2+2\gamma\langle B_{\xi_k}\omega_k-B_{\xi_k}y_k , y_k-x^*\rangle.
					\end{aligned}  \label{24}
				\end{equation}
				and
				\begin{equation}
					\begin{aligned}
					   & 2\langle u_k-u_{k+1}, y_k-x^*\rangle \\
                       =&2\langle u_k,x_k-x^*\rangle+2\langle u_k,y_k-x_k\rangle-2\langle u_{k+1},x_{k+1}-x^*\rangle \\
                       &-2\gamma\langle u_{k+1},S^{-1}B_{\xi_k}y_k-S^{-1}B_{\xi_k}x_k\rangle. \\
                       \leq &2\langle u_k,x_k-x^*\rangle+2\|u_k\|_{S^{-1}}\|y_k-x_k\|_S-2\langle u_{k+1},x_{k+1}-x^*\rangle \\
                       & +2\gamma\|u_{k+1}\|_{S^{-1}}\|B_{\xi_k}y_k-B_{\xi_k}x_k\|_{S^{-1}} \\
                       \leq&  2\langle u_k,x_k-x^*\rangle-2\langle u_{k+1},x_{k+1}-x^*\rangle+L_{k-1}\|y_{k-1}-\bar{x}_{k-1}\|_S^2 \\
                       & +L_{k-1}\|y_{k}-x_{k}\|_S^2+2\gamma\|u_{k+1}\|_{S^{-1}}\|B_{\xi_k}y_k-B_{\xi_k}x_k\|_{S^{-1}}. \\
					\end{aligned}  \label{25}
				\end{equation}
By Lemma \ref{lem0}, we get
				\begin{equation}
					-2\gamma\langle C\omega_k-Cx^* , y_k-x^* \rangle \leq \frac{\gamma\beta}{2}\|y_k-\omega_k\|_S^2. \label{26}
				\end{equation}
				Combining \eqref{24}, \eqref{25},  and \eqref{26}, we get
				\begin{equation}
					\begin{aligned} \label{27}
						& 2\gamma\langle B\omega_k-Bx^*+ B_{\xi_k}y_k-B_{\xi_k}\omega_k , y_k-x^*\rangle+\|x_{k+1}-x^*\|_S^2+2\langle u_{k+1}, x_{k+1}-x^*\rangle\\
						\leq &  \lambda\|x_k-x^*\|_S^2 +(1-\lambda)\|\omega_k-x^*\|_S^2+2\langle u_k,x_k-x^*\rangle+L_{k-1}\|y_{k-1}-\bar{x}_{k-1}\|_S^2 \\
                        &+\|x_{k+1}-y_k\|_S^2-(\lambda-L_{k-1})\|y_k-x_k\|_S^2 -(1-\lambda)\|y_k-\omega_k\|^2\\
                        &+2\gamma\|u_{k+1}\|_{S^{-1}}\|B_{\xi_k}y_k-B_{\xi_k}x_k\|_{S^{-1}}+\frac{\gamma\beta}{2}\|y_k-\omega_k\|_S^2.
					\end{aligned} 					
				\end{equation}
Taking expectation $\mathbb{E}_k$ on \eqref{27} and using
\begin{equation*}
\mathbb{E}_k [ \langle B\omega_k+B_{\xi_k}y_k-B_{\xi_k}\omega_k, y_k-x^* \rangle]=\langle By_k, y_k-x^* \rangle ,
\end{equation*}
we obtain
\begin{equation}
					\begin{aligned} \label{28}
						& 2\gamma\langle By_k-Bx^* , y_k-x^*\rangle+ \mathbb{E}_k \|x_{k+1}-x^*\|_S^2+2\mathbb{E}_k\langle u_{k+1}, x_{k+1}-x^*\rangle\\
						\leq &  \lambda\|x_k-x^*\|_S^2 +(1-\lambda)\|\omega_k-x^*\|_S^2+2\langle u_k,x_k-x^*\rangle+L_{k-1}\|y_{k-1}-\bar{x}_{k-1}\|_S^2 \\
                        &+\mathbb{E}_k\|x_{k+1}-y_k\|_S^2-(\lambda-L_{k-1})\|y_k-x_k\|_S^2 -(1-\lambda)\|y_k-\omega_k\|^2\\
                        &+2\gamma\mathbb{E}_k[\|u_{k+1}\|_{S^{-1}}\|B_{\xi_k}y_k-B_{\xi_k}x_k\|_{S^{-1}}]+\frac{\gamma\beta}{2}\|y_k-\omega_k\|_S^2.
					\end{aligned} 					
				\end{equation}
By the monotonicity of $B$, we have
\begin{equation}
\label{29}
2\gamma\langle By_k-Bx^* , y_k-x^*\rangle \geq 0.
\end{equation}
Combining the definition of $x_{k+1}$ and \eqref{lip}, we have
\begin{equation*}
\label{30}
\mathbb{E}_k \|x_{k+1}-y_k\|_S^2 \leq \gamma^2\theta^2\|y_k-\omega_k\|_S^2.
\end{equation*}
With the help of the Young’s inequality, we get
\begin{equation*}
2\gamma\mathbb{E}_k[\|u_{k+1}\|_{S^{-1}}\|B_{\xi_k}y_k-B_{\xi_k}x_k\|_{S^{-1}}] \leq \gamma L_k\theta\|y_k-\bar{x}_k\|_S^2+\gamma L_k\theta\|y_k-x_k\|_S^2.
\end{equation*}
Further, we obtain
\begin{equation}
\begin{aligned}
\label{31}
&\mathbb{E}_k \|x_{k+1}-x^*\|_S^2+2\mathbb{E}_k\langle u_{k+1}, x_{k+1}-x^*\rangle \\
\leq &  \lambda\|x_k-x^*\|_S^2 +(1-\lambda)\|\omega_k-x^*\|_S^2+2\langle u_k,x_k-x^*\rangle+L_{k-1}\|y_{k-1}-\bar{x}_{k-1}\|_S^2 \\
&-(\lambda-L_{k-1}-\gamma L_k\theta)\|y_k-x_k\|_S^2 -(1-\lambda-\gamma^2\theta^2-\frac{\gamma\beta}{2})\|y_k-\omega_k\|_S^2\\
&+\gamma L_k\theta\|y_k-\bar{x}_k\|_S^2.\\
\end{aligned} 					
\end{equation}
On the one hand, the definition of $w_{k+1}$ and $\mathbb{E}_{k+\frac{1}{2}}$ yield that
\begin{equation}
\label{32}
\frac{1-\lambda}{p}\mathbb{E}_{k+\frac{1}{2}}\|\omega_{k+1}-x^*\|_S^2=(1-\lambda)\|x_{k+1}-x^*\|_S^2+(1-\lambda)\frac{1-p}{p}\|\omega_k-x^*\|_S^2.
\end{equation}
Then apply to \eqref{32} the tower property $\mathbb{E}_k[\mathbb{E}_{k+\frac{1}{2}}[\cdot]]=\mathbb{E}_k[\cdot]$, we have
\begin{equation}
\label{33}
\frac{1-\lambda}{p}\mathbb{E}_{k}\|\omega_{k+1}-x^*\|_S^2=(1-\lambda)\mathbb{E}_{k}\|x_{k+1}-x^*\|_S^2+(1-\lambda)\frac{1-p}{p}\|\omega_k-x^*\|_S^2.
\end{equation}
Adding \eqref{33} to \eqref{31}, we obtain
\begin{equation}
\begin{aligned}
\label{34}
&\lambda\mathbb{E}_k \|x_{k+1}-x^*\|_S^2+\frac{1-\lambda}{p}\mathbb{E}_k \|\omega_{k+1}-x^*\|_S^2+2\mathbb{E}_k\langle u_{k+1}, x_{k+1}-x^*\rangle \\
\leq &  \lambda\|x_k-x^*\|_S^2 +\frac{1-\lambda}{p}\|\omega_k-x^*\|_S^2+2\langle u_k,x_k-x^*\rangle+L_{k-1}\|y_{k-1}-\bar{x}_{k-1}\|_S^2 \\
&-(\lambda-L_{k-1}-\gamma L_k\theta)\|y_k-x_k\|_S^2 -(1-\lambda-\gamma^2\theta^2-\frac{\gamma\beta}{2})\|y_k-\omega_k\|_S^2\\
&+\gamma L_k\theta\|y_k-\bar{x}_k\|_S^2.\\
\end{aligned} 					
\end{equation}
On the other hand,
\begin{equation}
\begin{aligned}
\label{35}
&(\gamma L_k\theta+L_k)\|y_k-\bar{x}_k\|_S^2 \\
=& (\gamma L_k\theta+L_k)\|\lambda (y_k-x_k)+(1-\lambda)(y_k-\omega_k)\|_S^2 \\
\leq & \lambda(\gamma L_k\theta+L_k)\|y_k-x_k\|_S^2+(1-\lambda)(\gamma L_k\theta+L_k)\|y_k-\omega_k\|_S^2.
\end{aligned} 					
\end{equation}
Therefore, by the definition of $\Gamma_k(x)$ in \eqref{Gammak}.
\begin{equation}
\begin{aligned}
\label{36}
\mathbb{E}_k[\Gamma_{k+1}(x^*)]
\leq & \Gamma_{k}(x^*)-(\lambda-L_{k-1}-\gamma L_k\theta-\lambda(\gamma L_k\theta+L_k))\|y_k-x_k\|_S^2 \\ &-(1-\lambda-\gamma^2\theta^2-\frac{\gamma\beta}{2}-(1-\lambda)(\gamma L_k\theta+L_k))\|y_k-\omega_k\|_S^2,
\end{aligned} 					
\end{equation}
showing \eqref{lypunov}. This completes the proof.
\end{proof}
\begin{theorem}
\label{theorem2}
{
\noindent
Suppose that Assumption \ref{assumption2} holds, $p \in (0,1]$, and there exists an $\epsilon >0$ such that
\begin{equation}
\label{assumption3}
\begin{aligned}
\begin{split}
&\lambda-L_{k-1}-\gamma L_k\theta-\lambda(\gamma L_k\theta+L_k) \geq \epsilon, \\
&1-\lambda-\gamma^2\theta^2-\frac{\gamma\beta}{2}-(1-\lambda)(\gamma L_k\theta+L_k) \geq \epsilon.
\end{split}
\end{aligned}
\end{equation}
Then the sequence $\{x_k\}_{k \in \mathbb{N}}$  generated by Algorithm \ref{algorithm5} converges weakly almost surely to a $\Omega^*$-valued random variable.
}
\end{theorem}
\begin{proof}
Leveraging the definition of $\Gamma_k(x^*)$ in \eqref{Gammak}, the Lipschitz property of $\gamma M_k-S$ and Cauchy-Schwarz inequality, we establish the non-negativity of $\Gamma_{k}(x^*)$ $(\forall k \in \mathbb{N}_0)$:
\begin{equation}
\label{37}
\begin{aligned}
\Gamma_{k}(x^*)=&\lambda\|x_k-x^*\|_{S}^2+\frac{1-\lambda}{p}\|\omega_k-x^*\|_S^2+2\langle u_k, x_k-x^*\rangle +L_{k-1}\|y_{k-1}-\bar{x}_{k-1}\|_S^2 \\
\geq & \lambda\|x^k-x^*\|_{S}^2+\frac{1-\lambda}{p}\|\omega_k-x^*\|_S^2-L_{k-1}(\|y_{k-1}-\bar{x}_{k-1}\|_S^2+\|x_k-x^*\|_S^2)\\
&+L_{k-1}\|y_{k-1}-x_{k-1}\|_S^2 \\
=& (\lambda-L_{k-1})\|x_k-x^*\|_S^2+\frac{1-\lambda}{p}\|\omega_k-x^*\|_S^2\geq 0,
\end{aligned}
\end{equation}
provided  $\lambda-L_{k-1} \geq 0$.

By Lemma \ref{Comettes} (i) and assumption \eqref{assumption3}, we know that there exists $\Xi \in \mathcal{F}$  such that $\mathbb{P}(\Xi)=1$ and $\forall  \theta \in \Xi$, it holds that $y_k(\theta)-x_k(\theta) \rightarrow 0$, $y_k(\theta)-\omega_k(\theta) \rightarrow 0$,  which means that $\omega_k(\theta)-\bar{x}_k(\theta) \rightarrow 0$. Moreover, by Lemma \ref{Comettes} (ii), there exists $\Xi^{'} \in \mathcal{F}$ such that $\mathbb{P}(\Xi^{'})=1$ and $\{\Gamma_k(x^*)(\theta)\}_{k\in \mathbb{N}_0}$ converges for $ \forall \theta \in \Xi^{'}$, $\forall x^* \in \Omega^*$. This further implies that the sequence $\{x_{k}(\theta)\}_{k \in \mathbb{N}}$ is bounded. Now,  pick any  $\theta \in \Xi \bigcap \Xi^{'}$, and consider a  weakly convergent subsequence $\{x_{k_j}(\theta)\}_{j \in \mathbb{N}}$ of $\{x_{k}(\theta)\}_{k \in \mathbb{N}}$. Without loss of generality, suppose that $x_{k_j}(\theta) \rightharpoonup \bar{x}(\theta)$.  It follows from $y_{k_j}(\theta)-x_{k_j}(\theta) \rightarrow 0$ that $y_{k_j}(\theta) \rightharpoonup \bar{x}(\theta)$.
From the inclusion given in \eqref{20}, we have
\begin{equation*}
M_k\bar{x}_{k_j}(\theta)-M_ky_{k_j}(\theta)-(B+C)\omega_{k_j}(\theta)+(B+C)y_{k_j}(\theta) +\gamma^{-1}u_{k_j}(\theta)\in (A+B+C)y_{k_j}(\theta).
\end{equation*}
Using the Lipschitz property of $M_k$, $B+C$ and $\gamma M_k - S$, we get
\begin{equation*}
M_k\bar{x}_{k_j}(\theta)-M_ky_{k_j}(\theta)-(B+C)\omega_{k_j}(\theta)+(B+C)y_{k_j}(\theta) +\gamma^{-1}u_{k_j}(\theta) \rightarrow 0.
\end{equation*}
Furthermore, by the assumption that the operator $B$ has full domain, Corollary 25.5 (i) in \cite{BC2011} ensures that $A+B$ is maximally monotone. Combining this with Lemma 2.1 in \cite{Showalter} and the assumption that $C$ is cocoercive, we conclude that $A+B+C$ is also maximally monotone. Then, by Lemma \ref{lem3}, we obtain $(\bar{x}(\theta) , 0) \in \gra(A+B+C)$, i.e., $\bar{x}(\theta) \in \Omega^*$. Hence, all weak cluster points of $\{x_{k}(\theta)\}_{k \in \mathbb{N}}$ belong to $\Omega^*$. Finally, by Lemma \ref{Comettes} (iii), the sequence $\{x_k\}_{k \in \mathbb{N}}$ converges weakly almost surely to a $\Omega^*$-valued random variable.
\end{proof}

\subsection{Linear convergence}
This section establishes the $R$-linear convergence of the sequence $\{x_k\}_{k \in \mathbb{N}}$ generated by Algorithm \ref{algorithm5}, under the assumption that the operator $B$ is  strongly monotone.
\begin{theorem}
\label{linear2}
{
\noindent
Let Assumption {\rm\ref{assumption2}} hold, and suppose that there exists a solution $x^*$ to \eqref{PROB2}. Assume further that the operator $B$ is $\rho$-strongly monotone.  Set $\lambda=1-p$, with $p \in (0,1)$, and choose parameters satisfying $L\geq \frac{3-p}{2}$, $\gamma=\min\{\frac{\sqrt{p}}{2\theta}, \frac{p}{\beta}, \frac{\alpha}{L_k \theta}\}$, $\varepsilon_3>0$, and $L_k+\alpha \leq \frac{1-\sqrt{p}}{4}$ with $L_k\geq 0,\ \alpha \geq 0$.
Then, for the sequence $\{x_k\}_{k \in \mathbb{N}}$ generated by Algorithm {\rm\ref{algorithm5}},  the following holds
\begin{equation}
\label{R-linear}
(1-p-L_k)\mathbb{E}\|x_{k+1}-x^*\|^2  \leq \frac{(1-p)\|x_1-x^*\|_S^2+\|\omega_1-x^*\|_S^2+d_1}{(1+c/(2L))^k},
\end{equation}
where $d_1= 2\langle u_1,x_1-x^*\rangle+L_0\|y_0-\bar{x}_0\|_S^2$ with constant $L \geq \frac{3-p}{2},$ and
\begin{equation}
	\begin{aligned}
		c=\min \big\{\frac{\gamma \mu}{1+\frac{L_k}{2L\varepsilon_3}}, &\frac{2L(1-p-L_{k-1}-\alpha-(1-p)(\alpha+L_k))}{(1-p) L_k(\varepsilon_3+1)},\\
		& \frac{Lp(1-\sqrt{p}-4(L_k+\alpha))}{2(1-p)(4+p)+2pL_k(\varepsilon_3+1)}\big\}.
	\end{aligned} 					
\end{equation}
}
\end{theorem}

\begin{proof}
Since $B$ is $\mu$-strongly monotone w.r.t. $S$, then \eqref{29} becomes
\begin{equation}
\label{38}
2\gamma\langle By_k-Bx^* , y_k-x^*\rangle \geq 2\gamma\mu\|y_k-x^*\|_S^2.
\end{equation}
Proceed as in the proof of Theorem \ref{theorem2} to obtain, in palce of \eqref{36},
\begin{equation}
\begin{aligned}
\label{39}
&\ 2\gamma\mu\|y_k-x^*\|_S^2+\lambda\mathbb{E}_k \|x_{k+1}-x^*\|_S^2+\frac{1-\lambda}{p}\mathbb{E}_k \|\omega_{k+1}-x^*\|_S^2 \\
&+2\mathbb{E}_k\langle u_{k+1}, x_{k+1}-x^*\rangle+L_{k}\|y_{k}-\bar{x}_{k}\|_S^2 \\
\leq &  \lambda\|x_k-x^*\|_S^2 +\frac{1-\lambda}{p}\|\omega_k-x^*\|_S^2+2\langle u_k,x_k-x^*\rangle+L_{k-1}\|y_{k-1}-\bar{x}_{k-1}\|_S^2 \\
&-(\lambda-L_{k-1}-\gamma L_k\theta}{-\lambda(\gamma L_k\theta+L_k))\|y_k-x_k\|_S^2 \\ &-(1-\lambda-\gamma^2\theta^2-\frac{\gamma\beta}{2}-(1-\lambda)(\gamma L_k\theta+L_k))\|y_k-\omega_k\|_S^2.
\end{aligned} 					
\end{equation}
By the inequality $\|a+b\|^2 \leq 2\|a\|^2 + 2\|b\|^2$, the definition of $x_{k+1}$, and  \eqref{lip}, we obtain
\begin{equation}
\label{40}
\aligned
2\gamma\mu\|y_k-x^*\|^2
& \geq \gamma \mu \mathbb{E}_k[\|x_{k+1}-x^*\|_S^2]-2\gamma\mu \mathbb{E}_k[\|\gamma(B_{\xi_k}\omega_k - B_{\xi_k}y_k)\|_{S^{-1}}^2] \\
& \geq \gamma \mu \mathbb{E}_k[\|x_{k+1}-x^*\|_S^2]-2\gamma^3\theta^2\mu\|y_k-\omega_k\|_S^2.
\endaligned
\end{equation}
Combining \eqref{39}, \eqref{40}, and the identity $\lambda=1-p$,  we obtain
\begin{equation}
\label{41}
\aligned
&(1-p+\gamma \mu)\mathbb{E}_k [\|x_{k+1}-x^*\|_S^2]+\mathbb{E}_k[\|\omega_{k+1}-x^*\|^2]+2\mathbb{E}_k\langle u_{k+1}, x_{k+1}-x^*\rangle \\
&+L_{k}\|y_{k}-\bar{x}_{k}\|_S^2 \\
\leq & (1-p)\|x_k-x^*\|_S^2 +\|\omega_k-x^*\|_S^2+2\langle u_k,x_k-x^*\rangle+L_{k-1}\|y_{k-1}-\bar{x}_{k-1}\|_S^2 \\
&-(1-p-L_{k-1}-\gamma L_k\theta-(1-p)(\gamma L_k\theta+L_k))\|y_k-x_k\|_S^2 \\
&-(p-\gamma^2\theta^2-\frac{\gamma\beta}{2}-p(\gamma L_k\theta+L_k)-2\gamma^3\theta^2\mu)\|y_k-\omega_k\|_S^2.\\
\endaligned
\end{equation}
On the one hand, using the assumptions $\gamma \leq \frac{\alpha}{L_k\theta}$, $0 < p \leq 1$, and $L_k+\alpha \leq \frac{1-\sqrt{p}}{4}$ for $\forall k \in \mathbb{N}$, we obtain the following inequality:
\begin{equation}
\aligned
&1-p-L_{k-1}-\gamma L_k\theta-(1-p)(\gamma L_k\theta+L_k)\\
\geq & 1-p-L_{k-1}-\alpha-(1-p)(\alpha+L_k) \geq 0.
\endaligned
\end{equation}
Besides, by $\gamma=\min\{\frac{\sqrt{p}}{2\theta}, \frac{p}{\beta}, \frac{\alpha}{L_k \theta}\}$, $\mu\leq \theta$, and $L_k+\alpha \leq \frac{1-\sqrt{p}}{4}$ for $\forall k \in \mathbb{N}$, the following inequality holds:
\begin{equation}
\aligned
&p-\gamma^2\theta^2-\frac{\gamma\beta}{2}-p(\gamma L_k\theta+L_k)-2\gamma^3\theta^2\mu \\
\geq & p-\frac{p}{4}-\frac{p}{2}-\alpha p-L_kp-\frac{p^{\frac{3}{2}}}{4} \\
= & \frac{p(1-\sqrt{p})}{4}-(L_k+\alpha)p \geq 0.
\endaligned
\end{equation}
Thus,
\begin{equation}
\label{41+}
\aligned
&(1-p+\gamma \mu)\mathbb{E}_k [\|x_{k+1}-x^*\|_S^2]+\mathbb{E}_k[\|\omega_{k+1}-x^*\|_S^2]+2\mathbb{E}_k\langle u_{k+1}, x_{k+1}-x^*\rangle \\
&+L_{k}\|y_{k}-\bar{x}_{k}\|_S^2 \\
\leq & (1-p)\|x_k-x^*\|_S^2 +\|\omega_k-x^*\|_S^2+2\langle u_k,x_k-x^*\rangle+L_{k-1}\|y_{k-1}-\bar{x}_{k-1}\|_S^2 \\
&-(1-p-L_{k-1}-\alpha-(1-p)(\alpha+L_k))\|y_k-x_k\|_S^2 \\
&-(\frac{p(1-\sqrt{p})}{4}-(L_k+\alpha)p)\|y_k-\omega_k\|_S^2.
\endaligned
\end{equation}
Set
\begin{equation*}
\aligned
d_k=&2\langle u_k,x_k-x^*\rangle+L_{k-1}\|y_{k-1}-\bar{x}_{k-1}\|_S^2, \\
e_k=&(1-p-L_{k-1}-\alpha-(1-p)(\alpha+L_k))\|y_k-x_k\|_S^2\\
&+(\frac{p(1-\sqrt{p})}{4}-(L_k+\alpha)p-\frac{c(1-p)(4+p)}{2L})\|y_k-\omega_k\|_S^2,
\endaligned
\end{equation*}
where $L\geq \frac{3-p}{2}$.
Then \eqref{41+} is equivalent to the following inequality
\begin{equation}
\label{41++}
\aligned
&(1-p+\gamma \mu)\mathbb{E}_k [\|x_{k+1}-x^*\|^2]+\mathbb{E}_k[\|\omega_{k+1}-x^*\|^2]+\mathbb{E}_k[d_{k+1}] \\
\leq & (1-p)\|x_k-x^*\|_S^2 +\|\omega_k-x^*\|_S^2+d_k-e_k-\frac{c(1-p)(4+p)}{2L}\|y_k-\omega_k\|_S^2. \\
\endaligned
\end{equation}
Using Lipschitzness of $u_{k+1}$ and Young's inequality, we have
\begin{equation}
\label{47}
\begin{aligned}
d_{k+1}
= & 2\langle u_{k+1},x_{k+1}-x^*\rangle+L_{k}\|y_k-\bar{x}_k\|_S^2 \\
\leq  & \varepsilon_3L_k\|y_{k}-\bar{x}_{k}\|_S^2+\frac{L_k}{\varepsilon_3}\|x_{k+1}-x^*\|_S^2+L_{k}\|y_{k}-\bar{x}_{k}\|_S^2\\
 = & L_k(\varepsilon_3+1)\|y_{k}-\bar{x}_{k}\|_S^2+\frac{L_k}{\varepsilon_3}\|x_{k+1}-x^*\|^2_S\\
 \leq &\lambda L_k(\varepsilon_3+1)\|y_{k}-x_{k}\|_S^2+(1-\lambda)L_k(\varepsilon_3+1)\|y_{k}-\omega_{k}\|_S^2 \\
& +\frac{L_k}{\varepsilon_3}\|x_{k+1}-x^*\|^2_S \\
= &(1-p) L_k(\varepsilon_3+1)\|y_{k}-x_{k}\|_S^2+pL_k(\varepsilon_3+1)\|y_{k}-\omega_{k}\|_S^2 \\
& +\frac{L_k}{\varepsilon_3}\|x_{k+1}-x^*\|^2_S.
\end{aligned} 					
\end{equation}
By the assumption that
\begin{equation}
\label{48}
\begin{aligned}
c=\min\big\{\frac{\gamma \mu}{1+\frac{L_k}{2L\varepsilon_3}}, &\frac{2L(1-p-L_{k-1}-\alpha-(1-p)(\alpha+L_k))}{(1-p) L_k(\varepsilon_3+1)},\\
 & \frac{Lp(1-\sqrt{p}-4(L_k+\alpha))}{2(1-p)(4+p)+2pL_k(\varepsilon_3+1)}\big\},
\end{aligned} 					
\end{equation}
and $L_k+\alpha \leq \frac{1-\sqrt{p}}{4}$ with $L_k\geq 0,\ \alpha \geq 0,$  we have that $c \geq 0.$ Then we get
\begin{equation}
\label{42-}
\begin{aligned}
\frac{c}{2L}d_{k+1}
\leq & (\gamma\mu-c)\|x_{k+1}-x^*\|_S^2+(1-p-L_{k-1}-\alpha-(1-p)(\alpha+L_k))\|y_k-x_k\|_S^2 \\
& +(\frac{p(1-\sqrt{p})}{4}-(L_k+\alpha)p-\frac{c(1-p)(4+p)}{2L})\|y_k-\omega_k\|_S^2 ,
\end{aligned} 	
\end{equation}
Taking expectation $\mathbb{E}_k$ on \eqref{42-} and combining \eqref{41++}, we obtain
\begin{equation}
	\label{80}
\aligned
&(1-p+c)\mathbb{E}_k [\|x_{k+1}-x^*\|_S^2]+\mathbb{E}_k[\|\omega_{k+1}-x^*\|_S^2]+(1+\frac{c}{2L})\mathbb{E}_k[d_{k+1}] \\
\leq & (1-p)\|x_k-x^*\|_S^2 +\|\omega_k-x^*\|_S^2+d_k-\frac{c(1-p)(4+p)}{2L}\|y_k-\omega_k\|_S^2. \\
\endaligned
\end{equation}
Similar to \eqref{40},
\begin{equation}
\label{42}
\aligned \frac{c}{L}\mathbb{E}_k[\|x_{k+1}-x^*\|_S^2]
& \geq \frac{c}{2L} \mathbb{E}_k[\|\omega_{k+1}-x^*\|_S^2] - \frac{c}{L}\mathbb{E}_k[\mathbb{E}_{k+\frac{1}{2}}\|x_{k+1}-\omega_{k+1}\|_S^2] \\
& = \frac{c}{2L} \mathbb{E}_k[\|\omega_{k+1}-x^*\|_S^2] - \frac{c(1-p)}{L}\mathbb{E}_k[\|x_{k+1}-\omega_k\|_S^2] \\
& \geq \frac{c}{2L} \mathbb{E}_k[\|\omega_{k+1}-x^*\|_S^2] -\frac{2c(1-p)(1+\gamma^2\theta^2)}{L}\|y_k-\omega_k\|_S^2 \\
& \geq \frac{c}{2L} \mathbb{E}_k[\|\omega_{k+1}-x^*\|_S^2] - \frac{c(1-p)(4+p)}{2L}\|y_k-\omega_k\|_S^2.
\endaligned
\end{equation}
Putting \eqref{42} into \eqref{80} gives us
\begin{equation}
\label{43}
\aligned
& (1-p+\frac{(L-1)c}{L})\mathbb{E}_k [\|x_{k+1}-x^*\|_S^2]+(1+\frac{c}{2L})\mathbb{E}_k[\|\omega_{k+1}-x^*\|_S^2] \\
&+(1+\frac{c}{2L})\mathbb{E}_k[d_{k+1}]\\
\leq & (1-p)\|x_k-x^*\|_S^2 +\|\omega_k-x^*\|_S^2+d_k.
\endaligned
\end{equation}
Then  using $1-p+\frac{(L-1)c}{L} \geq (1-p)(1+\frac{c}{2L})$ and taking the full expectation on \eqref{43}, yields
\begin{equation}
\label{44}
\aligned
&(1+\frac{c}{2L})\mathbb{E}[(1-p)\|x_{k+1}-x^*\|_S^2+\|\omega_{k+1}-x^*\|_S^2+ d_{k+1}] \\
\leq & \mathbb{E}[(1-p)\|x_k-x^*\|_S^2+\|\omega_k-x^*\|_S^2+d_k].
\endaligned
\end{equation}
In addition, by \eqref{37}, we know that
\begin{equation}
\label{45}
\aligned
&(1-p)\|x_{k+1}-x^*\|_S^2+\|\omega_{k+1}-x^*\|_S^2+d_{k+1} \\
\geq &(1-p-L_{k})\|x_{k+1}-x^*\|_S^2+\|\omega_{k+1}-x^*\|_S^2 \\
\geq &(1-p-L_{k})\|x_{k+1}-x\|_S^2.
\endaligned
\end{equation}
Taking the full expectation $\mathbb{E}$ on \eqref{45} and combining \eqref{44}, we obtain
\begin{equation}
\aligned
\label{46}
& (1-p-L_{k})\mathbb{E}\|x_{k+1}-x\|_S^2\\
\leq & \mathbb{E}[(1-p)\|x_{k+1}-x^*\|_S^2+\|\omega_{k+1}-x^*\|_S^2+ d_{k+1}] \\
\leq & \frac{\mathbb{E}[(1-p)\|x_k-x^*\|_S^2+\|\omega_k-x^*\|_S^2+d_k]}{1+c/(2L)}\\
\leq & \ldots \leq  \frac{(1-p)\|x_1-x^*\|_S^2+\|\omega_1-x^*\|_S^2+d_1}{(1+c/(2L))^k},
\endaligned
\end{equation}
which implies that \eqref{R-linear}. This completes the proof.
\end{proof}

\section{Numerical experiments}	
In this section, we focus on a special case of Algorithm \ref{algorithm1}, namely Algorithm \eqref{4operator}. We evaluate the performance of Algorithm \eqref{4operator} by comparing it with the FBHF splitting algorithm \eqref{FBHF}. All implementations were carried out in MATLAB, and the experiments were conducted on a PC equipped with an Intel Core i7-11800H CPU 2.30GHz with 32GB RAM, running Windows 11.

\subsection{Nonlinear constrained optimization problems}	
	{\rm
In this subsection, we consider the following nonlinear constrained optimization problem, which constitutes an important application of the monotone inclusion problem \eqref{ABC}:
\begin{equation}
	\label{fh}
	\min_{x \in C}f(x)+h(x),
\end{equation}
where $C= \{ x \in \mathcal{H} | ( \forall i \in \{1,...,q\}) \ g_i(x) \leq 0\} $, $f: \mathcal{H} \rightarrow ( -\infty ,+\infty]$ is a proper, convex  and  lower semi-continuous function, for every $i \in \{1,...,q\}$, $g_i : \dom(g_i) \subset \mathcal{H} \rightarrow \mathbb{R}$, and $h: \mathcal{H} \rightarrow \mathbb{R}$ are $C^1$ convex functions in $\nt \dom g_i$ and $\mathcal{H}$, respectively, and $\nabla h$ is $\beta$-Lipschitz. The solution to problem \eqref{fh}  corresponds to the saddle points of the Lagrangian
\begin{equation*}
	L(x,u)=f(x)+h(x)+u^{\top}g(x)-\iota_{\mathbb{R}_+^q}(u),
\end{equation*}
where $\iota_{\mathbb{R}_+^q}$ denotes the indicator function of $\mathbb{R}_+^q$. Under  standard qualification conditions, the saddle points can be characterized as solutions to the following  monotone inclusion problem \cite{{RT},{FBHF}}: find $x \in Y$ such that $\exists u \in \mathbb{R}_{+}^q$,
\begin{equation}
	\label{ABCxu}
	(0,0) \in A(x,u)+B(x,u)+C(x,u),
\end{equation}
where $Y \subset \mathcal{H}$ is a nonempty closed convex set modeling the prior information of the solution. The operators are defined as follows: $A:(x,u)\mapsto\partial f(x)\times{N_{\mathbb{R}_+^q}u}$, which is maximally monotone; $B:(x,u)\mapsto(\sum_{i=1}^q u_i \nabla g_i(x),-g_1(x),...,-g_q(x))$, which is nonlinear, monotone and continuous, and $C:(x,u)\mapsto(\nabla h(x),0)$ is $\frac{1}{\beta}$-cocoercive.

Let  $\mathcal{H}=\mathbb{R}^N$, $f=\iota_{[0,1]^N}$, $g_i(x)=d_i^{\top}x$ ($\forall i \in  \{1,...,q\}$) with $d_1,\ldots,d_q\in \mathbb{R}^N$, and  $h(x)=\frac{1}{2}\|Gx-b\|^2$ with  $G$ being an $m \times N$ real matrix, $N=2m$, $b\in \mathbb{R}^m$. Then the  operators in \eqref{ABCxu} become
		\begin{equation}
			\label{A1}
			\aligned
			&A:(x,u)\mapsto\partial{\iota_{[0,1]^N}(x)}\times{N_{\mathbb{R}_+^q}u},\\
			& B:(x,u)\mapsto(D^\top u,-Dx),\\
			& C:(x,u)\mapsto(G^\top(Gx-b),0),\\
			\endaligned
		\end{equation}
		where  $x\in \mathbb{R}^N$, $u\in \mathbb{R}_+^q$, $D=[d_1,\ldots,d_q]^\top$. It is easy to see that the operator $A$ is a maximally monotone operator, $B$ is a $L$-Lipschitz operator with $L=\|D\|$, and $C$ is a $\frac{1}{\beta}$-cocoercive operator with $\beta =\|G\|^{2}$. Therefore, problem \eqref{ABCxu} can be solved by  FBHF splitting algorithm \eqref{FBHF}.
		
On the other hands, we decompose $B$ into $B_1+B_2$, where
			\begin{equation}
			\label{B1B2}
			\aligned
			& B_1:(x,u)\mapsto(\frac{1}{2}D^\top u,-\frac{1}{2}Dx),\\
			& B_2:(x,u)\mapsto(\frac{1}{2}D^\top u,-\frac{1}{2}Dx).\\
			\endaligned
		\end{equation}
We know that $A+B_1$ is maximally monotone, $B_1$ and $B_2$ are both $\frac{L}{2}$-Lipschitz operator with $L=\|D\|$. Therefore, problem \eqref{ABCxu} can also be solved by  Algorithm \eqref{4operator}.

In the numerical experiments, the matrices $G,D$, and the vector $b$, as well as initial values $(x_{0},u_{0})$,  are all randomly generated. In the FBHF,   we set $\gamma_k=0.9\chi$ where $\chi=\frac{4}{\beta+\sqrt{\beta^2+16L^2}}$. For Algorithm \eqref{4operator}, the step-size is chosen as $\gamma_k=0.9 \bar{\gamma}$, with $\bar{\gamma}=\frac{(-\beta/2+2L)+\sqrt{(\beta/2+2L)^2+12L^2}}{6L^2}$. We use
$$
E_k=\frac{\|(x_{k+1}-x_k,u_{k+1}-u_k)\|}{\|(x_k,u_k)\|}<10^{-6},
$$
as the stopping criterion. We test eight problem sizes, each with ten randomly generated instances. The average number of iterations (denoted as ``av.iter") and the average CPU time in seconds (denoted as ``av.time") over the ten instances are reported in Table \ref{results-1}. The results demonstrate that Algorithm \eqref{4operator} outperforms FBHF in both iteration count and computational time.
\begin{table}[h!]
	\centering
	\footnotesize
	\renewcommand\arraystretch{1.5}
	\setlength\tabcolsep{5pt}
	\caption{Computational results with FBHF and Algorithm \eqref{4operator}}
	\begin{tabular}{c c c c c c}
		\hline
		\multicolumn{2}{c}{\textbf{Problem size}} &
		\multicolumn{2}{c}{\textbf{av.iter}} &
		\multicolumn{2}{c}{\textbf{av.time (s)}} \\
		\cline{1-2} \cline{3-4} \cline{5-6}
		$N$ & $q$ & FBHF & Algorithm \eqref{4operator} & FBHF & Algorithm \eqref{4operator} \\
		\hline
		\multirow{4}{*}{2000}
		& 100  & 4953.2    & \textbf{4309}    & 15.627   & \textbf{8.4047} \\
		& 200  & 4187.6  & \textbf{3411.8}  & 20.841    & \textbf{7.9688} \\
		& 500  & 3265.9    & \textbf{2620.4}  & 46.917    & \textbf{14.627} \\
		& 1000 & 2601.5  & \textbf{1882.2}  & 120.76    & \textbf{34.02} \\
		\hline
		\multirow{3}{*}{4000}
		& 100  & 6572.5    & \textbf{4891.1}    & 209.93   & \textbf{122.42} \\
		& 200  & 5320.6  & \textbf{3603}  & 225.48    & \textbf{107.37} \\
		& 500  & 3870.8  & \textbf{2564.2}  & 390.74    & \textbf{107.65} \\
		& 1000  & 2801.2  & \textbf{1977.5}  & 540.77   & \textbf{123.37} \\
		\hline
	\end{tabular} \label{results-1}
\end{table}
\subsection{Portfolio optimization problem}	
This subsection focuses on the classical mean-variance model for minimizing portfolio risk. Given a portfolio consists of $n$ different assets, and the rate of return of asset $i$ is a random variable with an expected value $m_i$. The objective is to determine the investment weights $x_i$ that minimize risk, subject to achieving a specified minimum expected portfolio return. Let $H$ represent the covariance matrix of the asset returns. The size of $H$ is $225\times 225$, and $\|H\|= 0.2263$. Specifically, the portfolio optimization problem is defined as follows:
\begin{equation}\label{portfolio-problem}
	\begin{aligned}
		\min_{x\in R^{225}} & \frac{1}{2}x^{T}Hx, \\
		 s.t. & \sum_{i=1}^{75} x_i \geq 0.3, \
		 \sum_{i=76}^{150} x_i \geq 0.3, \
		 \sum_{i=151}^{225} x_i \geq 0.3, \\
		 & \sum_{i=1}^{225}m_i x_i \geq r, \ \sum_{i=1}^{225}x_i = 1, 0\leq x_i \leq 1, i = 1, \cdots, 225.
	\end{aligned}
\end{equation}
Define $g(x)=Dx+b=$
\begin{equation}
\left[
\begin{array}{ccccccccc}
-m_1 & \cdots & -m_{75} & -m_{76} & \cdots & -m_{150} & -m_{151} & \cdots & -m_{225} \\
-1 & \cdots & -1 & 0 & \cdots & 0 & 0 & \cdots & 0\\
0 & \cdots & 0 & -1 & \cdots & -1 & 0 & \cdots & 0\\
0 & \cdots & 0 & 0 & \cdots & 0 & -1 & \cdots & -1
	\end{array}
	\right]
	\left[
	\begin{array}{cccc}
		x_1\\
		x_2\\
		\vdots\\
		x_{225}
	\end{array}
	\right]
	+\left[
	\begin{array}{cccc}
		r\\
		0.3\\
		0.3\\
		0.3
	\end{array}
	\right]
\end{equation}
and $C_2 = \{x\in R^n | \sum_{i=1}^{n}x_i = 1,\ 0\leq x_i \leq 1, i = 1, \cdots, n\}$, then the solution to the optimization problem \eqref{portfolio-problem} can be found via the saddle points of the Lagrangian function
\begin{equation*}
	L(x,u)=\frac{1}{2}x^{T}Hx+\iota_{C_2}(x)+u^{\top}g(x)-\iota_{\mathbb{R}_+^4}(u).
\end{equation*}
Under some standard qualifications, it can be obtained by solving a monotone inclusion problem: find $x \in \mathbb{R}^{225}$ such that $\exists u \in \mathbb{R}_{+}^4$,
\begin{equation}
	(0,0) \in A(x,u)+B(x,u)+C(x,u),
\end{equation}
where
  \begin{equation}
	\label{A2}
	\aligned
	&A:(x,u)\mapsto\partial{\iota_{C_2}(x)}\times{N_{\mathbb{R}_+^4}u},\\
	&B:(x,u)\mapsto(D^\top u,-Dx-b),\\
	&C:(x,u)\mapsto(Hx,0).\\
	\endaligned
\end{equation}
It is easy to see that the operator $A$ is a maximally monotone operator, $B$ is a $L$-Lipschitz operator with $L=\|D\|$, and $C$ is a $\frac{1}{\beta}$-cocoercive operator with $\beta =\|H\|$. Therefore, problem \eqref{portfolio-problem} can be solved by  FBHF splitting algorithm \eqref{FBHF}. Similar to \eqref{B1B2}, we can decompose $B$ into $B_1+B_2$. Thus, problem \eqref{ABCxu} can also be solved by  Algorithm \eqref{4operator}.

This study investigates the portfolio optimization problem \eqref{portfolio-problem}, as introduced in the MATLAB optimization case library, using the benchmark dataset from OR-Library \cite{ChangCOR2000}. The asset returns $m_i$ in this dataset range from $-0.008489$ to $0.003971$. In our experiments, we consider three target return level $r=0.001, 0.002$, and $0.003$.  The obtained results are presented in Table \ref{results-2}, which shows that the compared algorithms achieve nearly the same accuracy. Furthermore, as illustrated in Figure \ref{solution-portfolio}, the computed solutions for $r \in \{0.001, 0.002, 0.003\}$ consistently exhibit sparsity, confirming the effectiveness of the proposed approach in promoting sparse portfolio selection.
\begin{table}[htbp]
	\small
	\centering
	\caption{The objective function values (Obj), number of iterations (Iter), and CPU time (in seconds) for the compared methods on the portfolio optimization problem.}
	\begin{tabular}{c|c|ccc}
		\hline
		$r$ &  Methods &  Obj   &  Iter  & CPU  \\
		\hline
		\hline
		\multirow{2}[1]{*}{0.001}
		& FBHF   & $1.6386e-4$        & $\textbf{189499}$       & $5.0618$       \\
		& Algorithm \eqref{4operator}   & $1.6386e-4$      & $202203$       & $\textbf{4.9347}$       \\
		\hline
		\multirow{2}[1]{*}{0.002}
		& FBHF   & $2.0097e-4$        & $\textbf{193721}$       & $\textbf{4.7194}$       \\
		& Algorithm \eqref{4operator}   & $2.0097e-4$        & $214505$       & $5.4613$       \\
		\hline
		\multirow{2}[1]{*}{0.003}
		& FBHF   & $2.7695e-4$        & $156838$       & $3.8788$       \\
		& Algorithm \eqref{4operator}   & $\textbf{2.7694e-4}$        & $\textbf{154192}$       & $\textbf{3.7049}$       \\
		\hline
	\end{tabular}\label{results-2}
\end{table}

\begin{figure}[h]
	\centering
	\setlength{\abovecaptionskip}{-3pt}
	\setlength{\belowcaptionskip}{-2pt} %
	\captionsetup[subfigure]{labelformat=simple, skip=2pt} %

	\subfigure[$r=0.001$]{
		\scalebox{0.3}{\includegraphics{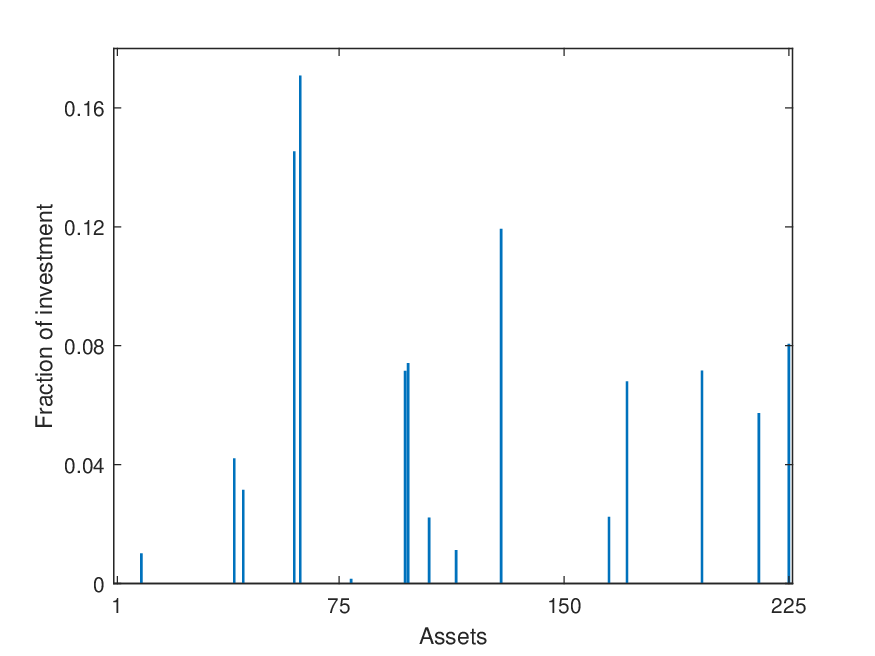}}
	} \hspace{1mm}
	\subfigure[$r=0.002$]{
		\scalebox{0.3}{\includegraphics{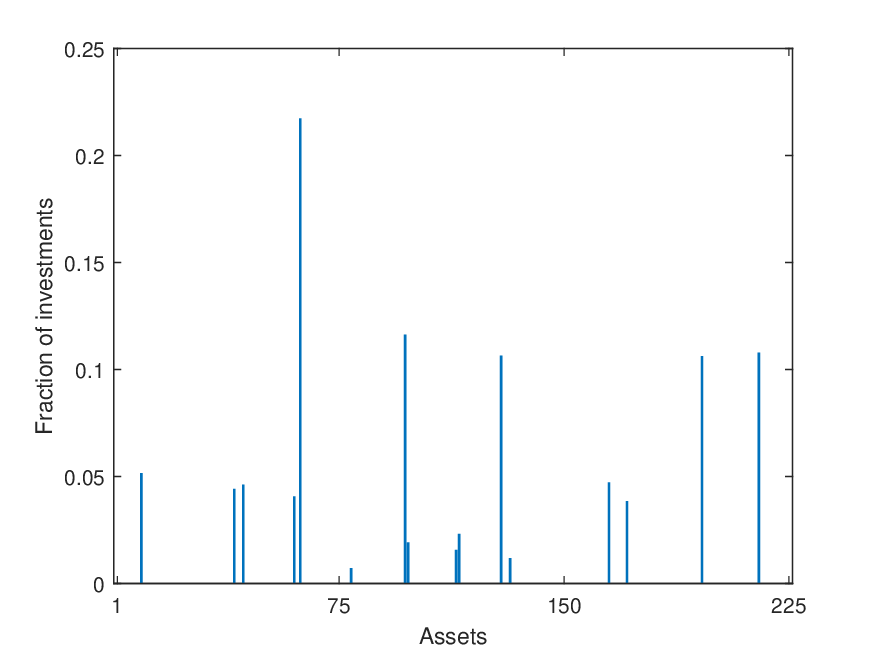}}
	} \hspace{1mm}
	\subfigure[$r=0.003$]{
		\scalebox{0.3}{\includegraphics{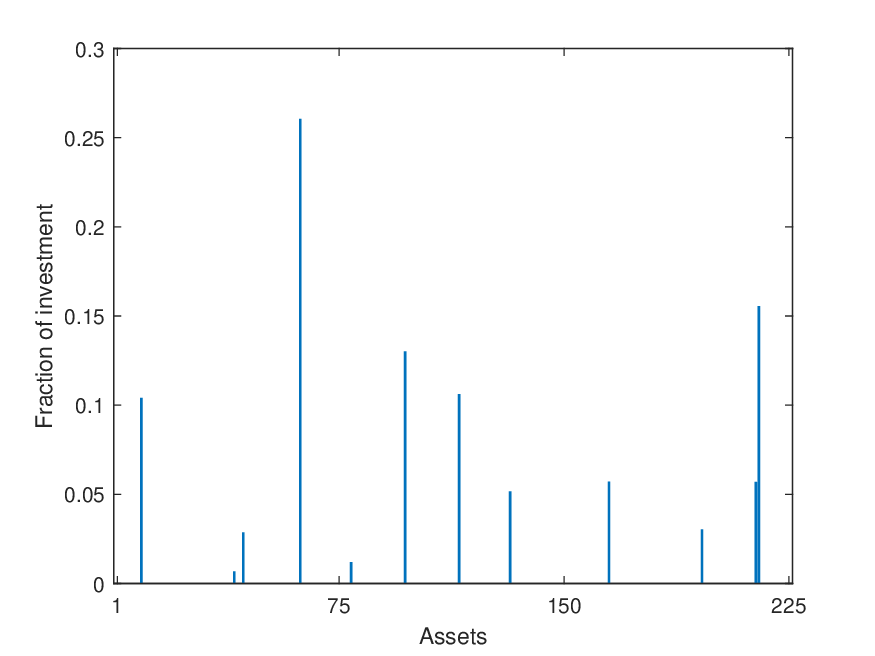}}
	}
	
	\vspace{3mm}

	\subfigure[$r=0.001$]{
		\scalebox{0.3}{\includegraphics{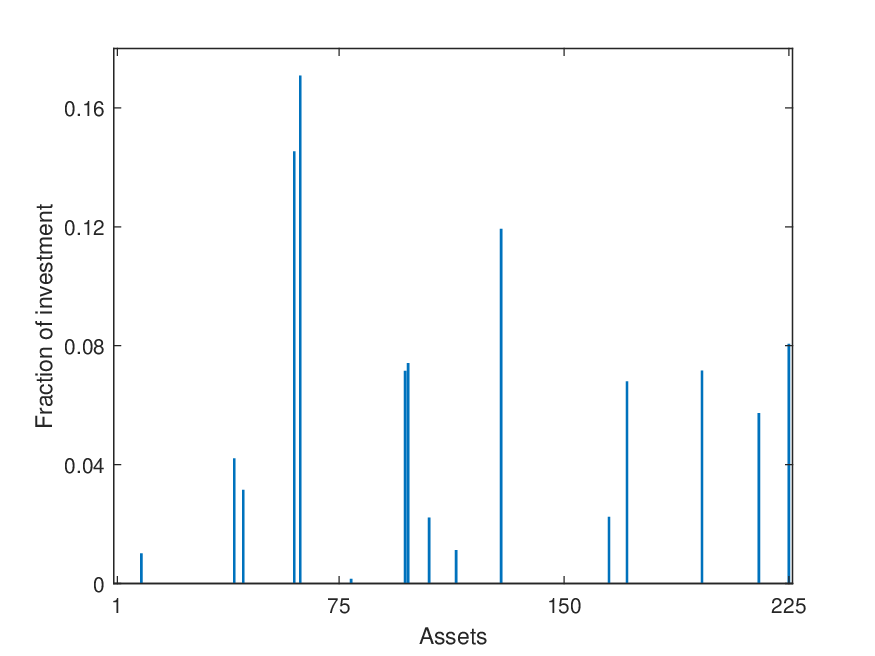}}
	} \hspace{1mm}
	\subfigure[$r=0.002$]{
		\scalebox{0.3}{\includegraphics{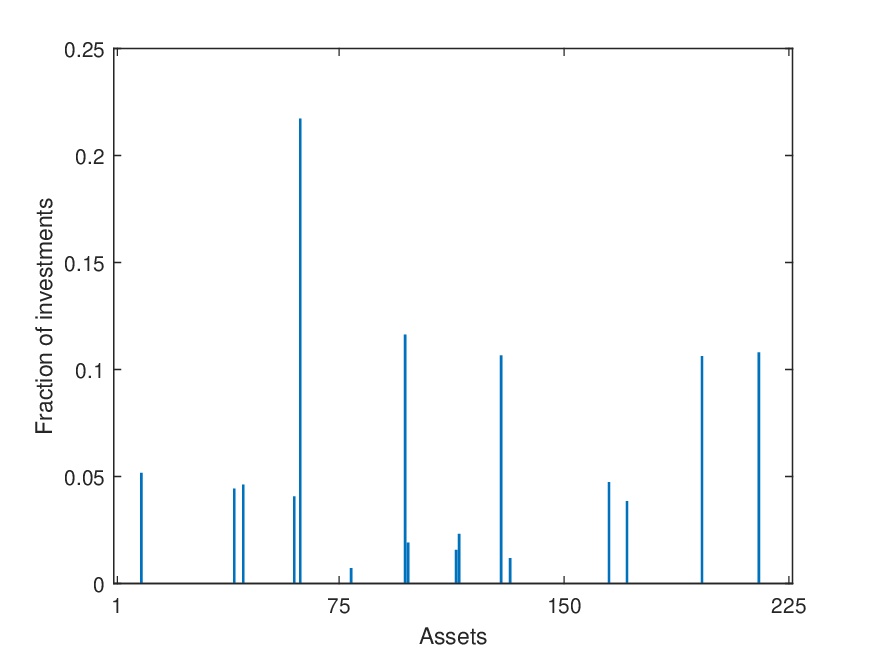}}
	} \hspace{1mm}
	\subfigure[$r=0.003$]{
		\scalebox{0.3}{\includegraphics{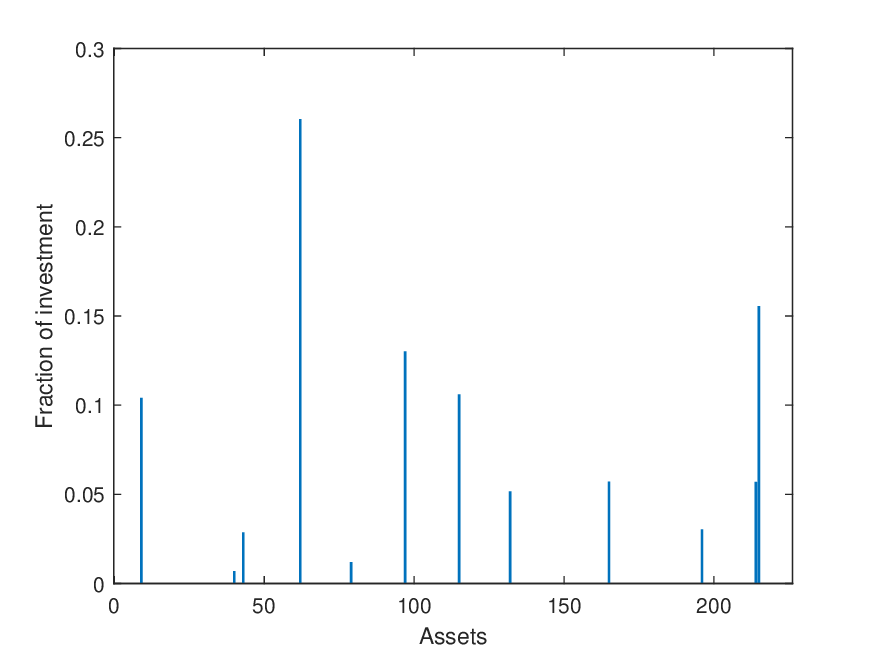}}
	}
	\vspace{3mm}
	\caption{Portfolio optimization solutions for problem \eqref{portfolio-problem}: Algorithm \eqref{4operator} (top row) versus FBHF splitting \eqref{FBHF} (bottom row).}
	\label{solution-portfolio}
\end{figure}

}

\section{Conclusions}
Monotone inclusion problems have emerged as a pivotal framework in applied mathematics, with significant applications across computational domains including signal processing, computer vision, and machine learning. In this paper, we introduced a novel splitting algorithm for solving \eqref{ABC}, which extended the FBHF splitting algorithm by incorporating a nonlinear momentum term. We established the weak convergence of the algorithm under appropriate step-size conditions and proved a linear convergence rate under the strong monotonicity assumption. We further developed a stochastic variance-reduced version of the algorithm with momentum, tailored for solving finite-sum monotone inclusion problems. Under standard conditions, we proved both weak almost sure convergence and linear convergence of the stochastic algorithm. Preliminary numerical experiments on nonlinear constrained optimization problems and portfolio optimization problems are presented to demonstrate the computational efficacy of the proposed algorithms.

\section*{Funding}

This work was supported  by the National Natural Science Foundations of China (12031003, 12571491, 12571558), the Guangzhou Education Scientific Research Project 2024 (202315829),  the Innovation Research Grant NO.JCCX2025018 for Postgraduate of Guangzhou University, and the Jiangxi Provincial Natural Science Foundation (20224ACB211004).

\section*{Competing Interests}

The authors declare no competing interests.

\section*{Data Availability Statement}

The data that support the findings of this study are available on request from the corresponding author upon reasonable request.



\end{document}